\documentclass[12pt]{amsart}
\usepackage{amsmath,amssymb,latexsym,esint,cite,mathrsfs}
\usepackage{verbatim,wasysym}
\usepackage[left=2.6cm,right=2.6cm,top=3cm,bottom=3cm]{geometry}
\makeatletter \@addtoreset{equation}{section} \makeatother

\usepackage{graphicx}
\usepackage{subfigure}
\usepackage{graphicx,epic,eepic}

\usepackage{color,enumitem,graphicx}
\usepackage[colorlinks=true,urlcolor=blue,
citecolor=red,linkcolor=blue,linktocpage,pdfpagelabels,
bookmarksnumbered,bookmarksopen]{hyperref}

\numberwithin{equation}{section}
\newtheorem{theorem}{Theorem}[section]
\newtheorem{lemma}[theorem]{Lemma}

\newtheorem{proposition}[theorem]{Proposition}
\newtheorem{remark}[theorem]{Remark}
\newtheorem{corollary}[theorem]{Corollary}
\numberwithin{equation}{section}

\allowdisplaybreaks

\begin{document}

\title[Weighted $p$-Laplace equations]
{Caffarelli-Kohn-Nirenberg-type inequalities related to weighted $p$-Laplace equations}

\author[S. Deng]{Shengbing Deng$^{\ast}$}
\address{\noindent Shengbing Deng (Corresponding author) \newline
School of Mathematics and Statistics, Southwest University,
Chongqing 400715, People's Republic of China}\email{shbdeng@swu.edu.cn}

\author[X. Tian]{Xingliang Tian} 
\address{\noindent Xingliang Tian  \newline
School of Mathematics and Statistics, Southwest University,
Chongqing 400715, People's Republic of China.}\email{xltian@email.swu.edu.cn}

\thanks{$^{\ast}$ Corresponding author}



\allowdisplaybreaks

\begin{abstract}
{\tiny We use a suitable transform related to Sobolev inequality to investigate the sharp constants and optimizers for some Caffarelli-Kohn-Nirenberg-type inequalities which are related to the weighted $p$-Laplace equations.
Moreover, we give the classification to the linearized problem related to the radial extremals. As an application, we investigate the gradient type remainder term of related inequality by using spectral estimate combined with a compactness argument which extends the work of Figalli and Zhang (Duke Math. J. 2022) at least for radial case.

    }
\end{abstract}

\vspace{3mm}

\maketitle

\section{{\bfseries Introduction and statement the main results}}\label{sectir}

    In this paper, we consider the classical radial solutions of the following weighted equation:
    \begin{equation}\label{Ppwh}
    -{\rm div}(|x|^{\alpha}|\nabla u|^{p-2}\nabla u)=|x|^{\beta} u^{p^*_{\alpha,\beta}-1} ,\quad u>0 \quad \mbox{in}\quad \mathbb{R}^N,
    \end{equation}
    where
    \begin{equation}\label{defpp}
    1<p<N,\quad p-N<\alpha<p+\beta,\quad p^*_{\alpha,\beta}=\frac{p(N+\beta)}{N-p+\alpha}.
    \end{equation}

    This problem, for $\alpha\neq 0$ or $\beta\neq 0$, generalizes the well-known equation which involves the critical Sobolev exponent
    \begin{equation}\label{bce}
    -{\rm div}(|\nabla u|^{p-2}\nabla u)=u^{p^*-1},\quad u>0\quad \mbox{in}\quad \mathbb{R}^N,
    \end{equation}
    where $p^*=\frac{Np}{N-p}$. Caffarelli et al. \cite{CGS89} proved that all the solutions are indeed the only ones of Aubin-Talenti functions \cite{Ta76} when $p=2$. Recently, V\'{e}tois \cite{Ve16} obtained the same conclusion for $1<p<2$ and Sciunzi \cite{Sc16} for $2<p<N$ however in \[\mathcal{D}^{1,p}_0(\mathbb{R}^N):=\{u\in L^{p^*}(\mathbb{R}^N): |\nabla u|\in L^p(\mathbb{R}^N)\},\]
    and it is not clear whether Aubin-Talenti functions have Moser index equal to one so far. Catino et al. \cite{CMR22} and Ou \cite{Ou22} removed the restricted condition $u\in\mathcal{D}^{1,p}_0(\mathbb{R}^N)$ for some special cases.
    We summarize those results as follows: let $1<p<N$ and $u\in\mathcal{D}^{1,p}_0(\mathbb{R}^N)$ be a solution to (\ref{bce}), then there exist $\lambda>0$ and $z\in\mathbb{R}^N$ such that
    \begin{equation*}
    u(x)=V_{\lambda,z}(x):=\frac{\gamma_{N,p}\lambda^{\frac{N-p}{p}}}
    {(1+\lambda^{\frac{p}{p-1}}|x-z|^{\frac{p}{p-1}})^{\frac{N-p}{p}}},\quad \gamma_{N,p}=\left[N\left(\frac{N-p}{p-1}\right)^{p-1}\right]^{\frac{N-p}{p^2}},
    \end{equation*}
    and they are only extremal functions (up to scalar multiplications) for the well-known Sobolev inequality
    \begin{equation}\label{bcesi}
    \int_{\mathbb{R}^N}|\nabla v|^p\geq S\left(\int_{\mathbb{R}^N}|v|^{p^*}\right)^\frac{p}{p^*},\quad \forall v\in \mathcal{D}^{1,p}_0(\mathbb{R}^N),
    \end{equation}
    for some constant $S>0$. Pistoia and Vaira \cite{PV21} proved that the solution $V\in \mathcal{D}^{1,p}_0(\mathbb{R}^N)$ of equation (\ref{bce}) is non-degenerate in the sense that all solutions of equation
    \begin{equation}\label{Ppwhlp}
    \begin{split}
    & -{\rm div}(|\nabla V|^{p-2}\nabla \varphi)-(p-2){\rm div}(|\nabla V|^{p-4}(\nabla V\cdot\nabla \varphi)\nabla V)
    =\left(p^*-1\right)V^{p^*-2}\varphi\quad \mbox{in}\ \mathbb{R}^N
    \end{split}
    \end{equation}
    in the space $\mathcal{D}^{1,2}_{0,*}(\mathbb{R}^N)$, are linear combination of the functions
    \begin{equation}\label{lpsy}
    Z_0(x)=\frac{N-p}{p}V+x\cdot \nabla V,\quad Z_i(x)=\frac{\partial V(x)}{\partial x_i},\quad i=1,\ldots,N.
    \end{equation}
    Here $\mathcal{D}^{1,2}_{0,*}(\mathbb{R}^N)$ is the weighted Sobolev space which is defined as the completion of $C^\infty_c(\mathbb{R}^N)$ with respect to the norm
    \[
    \|\varphi\|_{\mathcal{D}^{1,2}_{0,*}(\mathbb{R}^N)}:=\left(\int_{\mathbb{R}^N}|\nabla V|^{p-2}|\nabla \varphi|^2dx\right)^{\frac{1}{2}}.
    \]
    Pistoia and Vaira proved that $\mathcal{D}^{1,2}_{0,*}(\mathbb{R}^N)\hookrightarrow L^{2}_{0,*}(\mathbb{R}^N)$ continuously, where $L^{2}_{0,*}(\mathbb{R}^N)$ is the set of measurable functions $\varphi: \mathbb{R}^N\to \mathbb{R}$ whose norm
    \[
    \|\varphi\|_{L^{2}_{0,*}(\mathbb{R}^N)}:=\left(\int_{\mathbb{R}^N}
    V^{p^*-2}\varphi^2dx\right)^{\frac{1}{2}}.
    \]
    Furthermore, Figalli and Neumayer \cite{FN19} proved that  $\mathcal{D}^{1,2}_{0,*}(\mathbb{R}^N)\hookrightarrow \hookrightarrow L^{2}_{0,*}(\mathbb{R}^N)$ compactly when $2\leq p<N$ then they showed the solutions of \eqref{Ppwhlp} in $L^{2}_{0,*}(\mathbb{R}^N)$ are linear combination of the functions
    $Z_0$ and $Z_i$ $(i=1,\ldots,N)$. Figalli and Zhang \cite{FZ22} proved that  $\mathcal{D}^{1,2}_{0,*}(\mathbb{R}^N)\hookrightarrow \hookrightarrow L^{2}_{0,*}(\mathbb{R}^N)$ compactly for all $1< p<N$ and the non-degenerate conclusion of \cite{FN19} also holds.

    If $p-N<\alpha=p+\beta$, (\ref{Ppwh}) is related to the well known weighted Hardy inequality:
    \begin{equation}\label{phi}
    \int_{\mathbb{R}^N}|x|^\alpha|\nabla u|^p dx\geq \left(\frac{N-p+\alpha}{p}\right)^p\int_{\mathbb{R}^N}|x|^{\alpha-p}|u|^pdx,\quad \forall u\in C^\infty_0(\mathbb{R}^N),
    \end{equation}
    where $[(N-p+\alpha)/p]^p$ is sharp and not attained which implies that (\ref{Ppwh}) does not exist solutions in this case, see \cite[Lemma 2.3]{DR01}. If $\alpha=0$, $p=N$ and $\beta>-N$, the problem
    \begin{equation*}\label{Ppwhe}
    -{\rm div}(|\nabla u|^{N-2}\nabla u)
    =|x|^{\beta} e^u\quad \mbox{in}\quad \mathbb{R}^N,
    \quad \int_{\mathbb{R}^N}|x|^{\beta} e^u<\infty,
    \end{equation*}
    has its own interest, see \cite{Es21,PT01}.

    Let $\alpha=-pa$, $\beta=-bh$, then for $1<p<N$, (\ref{Ppwh}) is equivalent to
    \begin{equation}\label{Ppwht}
    -{\rm div}(|x|^{-pa}|\nabla u|^{p-2}\nabla u)=|x|^{-bh} u^{h-1} ,\quad u>0 \quad \mbox{in}\quad \mathbb{R}^N,
    \end{equation}
    where
    \begin{equation*}
    -\infty<a<\frac{N-p}{p},\quad a-\frac{N-p}{p}<b<a+1,\quad h=\frac{Np}{N-p(1+a-b)},
    \end{equation*}
    which is related the classical Caffarelli-Kohn-Nirenberg (we write (CKN) for short) inequality \cite{CKN84}:
    \begin{equation}\label{cknpi}
    \||x|^{-a}\nabla u\|_{L^p(\mathbb{R}^N)}\geq C\||x|^{-b}u\|_{L^h(\mathbb{R}^N)},\quad \forall u\in C^\infty_0(\mathbb{R}^N).
    \end{equation}
    For equation (\ref{Ppwht}) with $p=2$, Dolbeault et al. \cite{DEL16} proved an optimal rigidity result in the range $2<h<\frac{2N}{N-2}$. Namely, assuming the integrability condition $\int_{\mathbb{R}^N}|x|^{-bh}u^hdx<\infty$, they showed this equation has a unique (therefore radial) nonnegative solution either for  
    \begin{equation*}
    0\leq a<\frac{N-2}{2}\quad\mbox{and}\quad a\leq b<a+1,
    \end{equation*}
    or for
    \begin{equation*}
    a<0\quad\mbox{and}\quad b^{(1)}_{FS}(a)\leq b<a+1,
    \end{equation*}
    where \[b^{(k)}_{FS}(a):=\frac{N}{2}\left[1+\frac{4k(N-2+k)}{(N-2-2a)^2}\right]^{-1/2}-\frac{N-2-2a}{2},\quad k\in \mathbb{N}^+.\]
    When $a<0$ and $a\leq b<b^{(1)}_{FS}(a)$, Felli and Schneider \cite{FS03} have previously shown that the best constant in (\ref{cknpi}) with $p=2$ is achieved by non-radial functions only and, as a byproduct, (\ref{Ppwht}) with $p=2$ has non-radial nonnegative solutions and uniqueness is broken. Moreover, if $b=b^{(k)}_{FS}(a)$ for some $k\in \mathbb{N}^+$, non-radial bifurcation occurs, see \cite{BCG21,DGG17,GGN13}. It is worth to mention that for $p\neq 2$, Ciraolo and Corso \cite{CC22} proved that the positive solutions to \eqref{Ppwht} are radial symmetric for some special case which extends the result of Dolbeault et al. \cite{DEL16} to a general $p$.
    For more recent results about the (CKN) inequality, see \cite{ADG22,AB17,CFL21,DLL18,LL17} and the references therein.

    Coming back to (\ref{Ppwh}), in present paper, we are mainly concerned the linearized problem related to the radial solutions of (\ref{Ppwh}).
    Define $\mathcal{D}^{1,p}_\alpha(\mathbb{R}^N)$ as the completion of $C^\infty_c(\mathbb{R}^N)$ with the norm
    \begin{equation}\label{defnormd1pa}
    \|\phi\|_{\mathcal{D}^{1,p}_\alpha}=\left(\int_{\mathbb{R}^N}|x|^{\alpha}|\nabla \phi|^pdx\right)^{\frac{1}{p}}.
    \end{equation}
    Then we quote the following result by \cite{CKN84}, the so-called Caffarelli-Kohn-Nirenberg inequality, which, in the radial case, 
    partly extends the Hardy-Littlewood-Sobolev inequality \cite{Li83}.
    \begin{theorem}\label{thmPbcm}
    Assume $1<p<N$ and $p-N<\alpha<p+\beta$. Let $u\in \mathcal{D}^{1,p}_\alpha(\mathbb{R}^N)$ be a radial function, then we have
    \begin{equation}\label{Ppbcm}
    \int_{\mathbb{R}^N}|x|^{\alpha}|\nabla u|^p dx \geq S^{rad}_p(\alpha,\beta)\left(\int_{\mathbb{R}^N}|x|^{\beta}|u|^{p^*_{\alpha,\beta}} dx\right)^{\frac{p}{p^*_{\alpha,\beta}}},
    \end{equation}
    for some positive constant $S^{rad}_p(\alpha,\beta)$. The explicit form of the best constant in (\ref{Ppbcm}) is
    \begin{equation}\label{sbc}
    S^{rad}_p(\alpha,\beta)
    =\left(\frac{p+\beta-\alpha}{p}\right)^{\frac{pN-p+(p-1)\beta+\alpha}{N+\beta}}
    \left(\frac{2\pi^{\frac{N}{2}}}{\Gamma(\frac{N}{2})}\right)^{\frac{p+\beta-\alpha}{N+\beta}}
    C_p\left(\frac{p(N+\beta)}{p+\beta-\alpha}\right),
    \end{equation}
    where \[C_p(K)=K\left(\frac{K-p}{p-1}\right)^{p-1}
    \left(\frac{\Gamma(K/p)\Gamma(1+K-K/p)}{\Gamma(K)\Gamma(1+K/2)}\frac{\Gamma(K/2)}{2}\right)^{p/K}.\] 
    Moreover the extremal radial functions which achieve $S^{rad}_p(\alpha,\beta)$ in (\ref{Ppbcm}) are given by
    \begin{equation}\label{pbcm}
    W_{\lambda,\alpha,\beta}(x)=\frac{A\lambda^{\frac{N-p+\alpha}{p}}}
    {(1+\lambda^{\frac{p+\beta-\alpha}{p-1}}|x|^{\frac{p+\beta-\alpha}{p-1}})^{\frac{N-p+\alpha}{p+\beta-\alpha}}},
    \end{equation}
    for any $A\in\mathbb{R}\backslash\{0\}$ and $\lambda>0$.
    \end{theorem}

    It is well known that the extremal functions of $S^{rad}_p(\alpha,\beta)$ in (\ref{Ppbcm}) are the ground state solutions of the weighted $p$-Laplace equation (\ref{Ppwh}). Thus, as the direct consequence of Theorem \ref{thmPbcm}, we obtain
    \begin{corollary}\label{thmPpwh}
    Assume $1<p<N$ and $p-N<\alpha<p+\beta$. Then problem (\ref{Ppwh}) in the space $\mathcal{D}^{1,p}_\alpha(\mathbb{R}^N)$
    has a radial solution of the form
    \begin{equation}\label{defeu}
    U_{\lambda,\alpha,\beta}(x)=\frac{C_{N,\alpha,\beta}\lambda^{\frac{N-p+\alpha}{p}}}
    {(1+\lambda^{\frac{p+\beta-\alpha}{p-1}}|x|^{\frac{p+\beta-\alpha}{p-1}})^{\frac{N-p+\alpha}{p+\beta-\alpha}}},
    \end{equation}
    with $\lambda>0$, where \[C_{N,\alpha,\beta}=\left[(N+\beta)\left(\frac{N-p+\alpha}{p-1}\right)^{p-1}\right]^{\frac{N-p+\alpha}{p(p+\beta-\alpha)}}.\]
    \end{corollary}

    Our second result concerns the study of the linearized problem related to (\ref{Ppwh}) at the function $U_{1,\alpha,\beta}$. This leads to study the problem:
    \begin{align}\label{Ppwhl}
    & -{\rm div}(|x|^{\alpha}|\nabla U_{1,\alpha,\beta}|^{p-2}\nabla v)-(p-2){\rm div}(|x|^{\alpha}|\nabla U_{1,\alpha,\beta}|^{p-4}(\nabla U_{1,\alpha,\beta}\cdot\nabla v)\nabla U_{1,\alpha,\beta}) \nonumber\\
    & =(p^*_{\alpha,\beta}-1)|x|^{\beta} U_{1,\alpha,\beta}^{p^*_{\alpha,\beta}-2}v \quad \mbox{in}\quad \mathbb{R}^N,\quad v\in L^2_{\beta,*}(\mathbb{R}^N).
    \end{align}
    Here, the weighted Sobolev space $L^2_{\beta,*}(\mathbb{R}^N)$ is defined as the completion of $C^\infty_c(\mathbb{R}^N)$ with respect to  the norm
    \begin{equation}\label{defd12*n}
    \|\varphi\|_{L^{2}_{\beta,*}(\mathbb{R}^N)}:=\left(\int_{\mathbb{R}^N}
    |x|^\beta U_{1,\alpha,\beta}^{p^*_{\alpha,\beta}-2}\varphi^2dx\right)^{\frac{1}{2}}.
    \end{equation}
    For the classification of solutions for above equation, see \cite{AGP99,Re90}  with $\alpha=\beta=0$ and $p=2$, \cite{GGN13} with $\alpha=0$ and $p=2$, \cite{BCG21,DGG17} with $p=2$, and \cite{FN19,FZ22,PV21} with $\alpha=\beta=0$.  Next theorem characterizes all the solutions to (\ref{Ppwhl}).

    \begin{theorem}\label{coroPpwhlp}
    Assume $1<p<N$ and $p-N<\alpha<p+\beta$.
    If
    \begin{equation}\label{npk}
    \left(\frac{p+\beta-\alpha}{p}\right)^2\left[\frac{p(N+\beta)}{p+\beta-\alpha}-1\right]=k(N-2+k), \quad \mbox{for some}\quad k\in\mathbb{N}^+,
    \end{equation}
    then the space of solutions of (\ref{Ppwhl}) has dimension $1+\frac{(N+2k-2)(N+k-3)!}{(N-2)!k!}$ and is spanned by
    \begin{equation}\label{defwp0i}
    W_0(x)=\frac{(p-1)-|x|^{\frac{p+\beta-\alpha}{p-1}}}{(1+|x|^{\frac{p+\beta-\alpha}{p-1}})^\frac{N+\beta}{p+\beta-\alpha}},\quad W_{k,i}(x)=\frac{|x|^{\frac{p+\beta-\alpha}{p(p-1)}}\Psi_{k,i}(x)}{(1+|x|^{\frac{p+\beta-\alpha}{p-1}})^\frac{N+\beta}{p+\beta-\alpha}},
    \end{equation}
    where $\{\Psi_{k,i}\}$, $i=1,\ldots,\frac{(N+2k-2)(N+k-3)!}{(N-2)!k!}$, form a basis of $\mathbb{Y}_k(\mathbb{R}^N)$, the space of all homogeneous harmonic polynomials of degree $k$ in $\mathbb{R}^N$.
    Otherwise the space of solutions of (\ref{Ppwhl}) has dimension one and is spanned by $W_0(x)$ and in this case we say the solution $U_{1,\alpha,\beta}$ of equation (\ref{Ppwh}) is non-degenerate.
    \end{theorem}

    \begin{remark}\label{remmr}\rm
    The key step of the proofs of the above three conclusions is the change of variables $r\mapsto r^{\frac{p}{p+\beta-\alpha}}$, which are mainly inspired by \cite{CG10}, see also \cite[Theorem A.1]{GGN13}. Recently, after completing this work, we find a paper of Su and Wang \cite{SW22} which gives analogous results as Theorem \ref{thmPbcm} by using same methods. For the completeness of this paper, we keep this result and our main conclusion is Theorem \ref{coroPpwhlp}.

    Note that in the case $\alpha=\beta=0$ we get $k=1$ and one gets back the well known result for the equation involving the critical Sobolev exponent, see \cite{PV21}. Furthermore, for all $\alpha\neq 0$ or $\beta\neq 0$ the solutions of problem (\ref{Ppwh}) are invariant for dilations but not for translations, and Theorem \ref{coroPpwhlp} highlights the new phenomenon that if (\ref{npk}) holds then there exist new solutions to (\ref{Ppwhl}) that ``replace'' the ones due to the translations invariance. Indeed, nonradial solutions might appear. When $p=2$, $\alpha=0$, $\beta=2$ and $N\geq 4$ is even, Gladiali, Grossi and Neves \cite{GGN13} constructed the nonradial solutions to equation (\ref{Ppwh}), that is, for any $a\in\mathbb{R}$, the functions
    \begin{equation*}
    u(x)=u(|x'|,|x''|)=C_{N,0,2}(1+|x|^4-2a(|x'|^2-|x''|^2)+a^2)^{-\frac{N-2}{4}},
    \end{equation*}
    form a branch of solutions to (\ref{Ppwh}) bifurcating from $U_{1,0,2}$, where $(x',x'')\in \mathbb{R}^N=\mathbb{R}^{\frac{N}{2}}\times\mathbb{R}^{\frac{N}{2}}$.
    \end{remark}

    As an application of Theorem \ref{coroPpwhlp}, enlightened by  \cite{BE91,BrL85}, we are concerned the remainder terms of inequality (\ref{Ppbcm}) in radial spaces
    \begin{equation}\label{defd12xr}
    \mathcal{D}^{1,p}_{\alpha,r}(\mathbb{R}^N)=\{u\in \mathcal{D}^{1,p}_{\alpha}(\mathbb{R}^N): u(x)=u(|x|)\}.
    \end{equation}
    For $p=2$, Wang and Willem \cite{WaWi03} obtained the remainder terms of (CKN) inequality, Wei and Wu \cite{WW22} established the stability of the profile decompositions to the (CKN) inequality and also gave the gradient type reminder terms. For $\alpha=\beta=0$, Cianchi et al. \cite{CFMP09} proved a stability version for every $1<p<N$, Figalli and Neumayer \cite{FN19} proved the gradient stability for the Sobolev inequality when $p\geq 2$, Neumayer \cite{Ne20} extended the result in \cite{FN19} to all $1<p<N$. It is worth to mention that very recently, Figalli and Zhang \cite{FZ22} obtained the sharp stability of critical points of the Sobolev inequality (\ref{bcesi}) for all $1<p<N$ which reads
    \[
    \frac{\|\nabla u\|_{L^p(\mathbb{R}^N)}}{\|u\|_{L^{p^*}(\mathbb{R}^N)}}-S^{\frac{1}{p}}
    \geq c_{N,p} \inf_{v\in \mathcal{M}_0}\left(\frac{\|\nabla u-\nabla v\|_{L^p(\mathbb{R}^N)}}{\|\nabla u\|_{L^p(\mathbb{R}^N)}}\right)^{\max\{2,p\}},\quad \forall v\in \mathcal{D}^{1,p}_0(\mathbb{R}^N),
    \]
    for some constant $c_{N,p}>0$, where $\mathcal{M}_0$ is the set of minimizers for Sobolev inequality (\ref{bcesi}), furthermore the exponent $\max\{2,p\}$ is sharp, and this can be understood as a weak form as Bianchi-Egnell type (see \cite{BE91})
    \[
    \int_{\mathbb{R}^N}|\nabla u|^p dx- S\left(\int_{\mathbb{R}^N}|u|^{p^*} dx\right)^{\frac{p}{p^*}}\geq c_{N,p} \inf_{v\in \mathcal{M}_0}\|\nabla u-\nabla v\|_{L^p(\mathbb{R}^N)}^{\max\{2,p\}}.
    \]

    As mentioned above, we will extend the work of Figalli and Zhang \cite{FZ22} at least in radial space.
    \begin{theorem}\label{thmprtp}
    Assume $2\leq p<N$ and $p-N<\alpha<p+\beta$. Then there exists constant $B_1=B_1(N,p,\alpha,\beta)>0$ such that for every $u\in \mathcal{D}^{1,p}_{\alpha,r}(\mathbb{R}^N)$, it holds that
    \[
    \int_{\mathbb{R}^N}|x|^{\alpha}|\nabla u|^p dx- S^{rad}_p(\alpha,\beta)\left(\int_{\mathbb{R}^N}|x|^{\beta}|u|^{p^*_{\alpha,\beta}} dx\right)^{\frac{p}{p^*_{\alpha,\beta}}}\geq B_1 {\rm dist}(u,\mathcal{M})^p.
    \]
    where $\mathcal{M}=\{cU_{\lambda,\alpha,\beta}: c\in\mathbb{R}, \lambda>0\}$ is a two-dimensional manifold (see $U_{\lambda,\alpha,\beta}$ as in (\ref{defeu})), and ${\rm dist}(u,\mathcal{M}):=\inf_{c\in\mathbb{R},\lambda>0}\|u-cU_{\lambda,\alpha,\beta}\|_{\mathcal{D}^{1,p}_{\alpha}(\mathbb{R}^N)}$.
    \end{theorem}

    \begin{theorem}\label{thmprtp2}
    Assume $1<p< 2\leq N$ and $p-N<\alpha<p+\beta$. Then there exists constant $B_2=B_2(N,p,\alpha,\beta)>0$ such that for every $u\in \mathcal{D}^{1,p}_{\alpha,r}(\mathbb{R}^N)$, it holds that
    \[
    \int_{\mathbb{R}^N}|x|^{\alpha}|\nabla u|^p dx- S^{rad}_p(\alpha,\beta)\left(\int_{\mathbb{R}^N}|x|^{\beta}|u|^{p^*_{\alpha,\beta}} dx\right)^{\frac{p}{p^*_{\alpha,\beta}}}\geq B_2 {\rm dist}(u,\mathcal{M})^2.
    \]
    \end{theorem}

    \begin{remark}\label{remp2c}\rm
    In this paper, to handle the general case $1<p<N$ and obtain the remainder terms, as stated in \cite{FZ22}, we need to consider three cases $1<p\leq\frac{2(N+\alpha)}{N+2+\beta}$, $\frac{2(N+\alpha)}{N+2+\beta}< p<2$, and $2\leq p<N$ by using different arguments.
    In fact, we note that $p^*_{\alpha,\beta}:=\frac{p(N+\beta)}{N-p+\alpha}\leq 2$ implies $p\leq \frac{2(N+\alpha)}{N+2+\beta}$, $p^*_{\alpha,\beta}>2$ implies $p>\frac{2(N+\alpha)}{N+2+\beta}$, respectively. Moreover, $\frac{2(N+\alpha)}{N+2+\beta}<2$ is equivalent to $\alpha<2+\beta$. Therefore, in order to prove Theorem  \ref{thmprtp2}, we will split our problem into two cases:
    \begin{itemize}
    \item[$(i)$]
    $1<p\leq\frac{2(N+\alpha)}{N+2+\beta}$;
    \item[$(ii)$]
    $\frac{2(N+\alpha)}{N+2+\beta}< p<2$,
    \end{itemize}
    due to $p-N<\alpha<p+\beta$, then $1<p<2$ implies  $\alpha<2+\beta$ and also $\frac{2(N+\alpha)}{N+2+\beta}<2$.
    The reason why we consider the above two cases is that, it needs some appropriate algebraic inequalities which requires to compare $p$ and $p^*_{\alpha,\beta}$ with $2$, see Lemmas \ref{lemui1p} and \ref{lemui1p*l}. However,  $2\leq p<N$ implies $p^*_{\alpha,\beta}> 2$, and the $\mathcal{D}^{1,p}_{\alpha}(\mathbb{R}^N)$ norm is stronger than any weighted $\mathcal{D}^{1,2}_{\alpha,*}(\mathbb{R}^N)$ norm (see \eqref{defd12*nb}), so that we can deal with this case directly.

    Furthermore,
    by using perturbation methods as in \cite{FZ22}, it is easy to verify that the exponents $p$ in Theorem \ref{thmprtp} and $2$ in Theorem \ref{thmprtp2} are sharp.
    \end{remark}

    The paper is organized as follows: In Section \ref{sectpmr}, we deduce the best constant $S^{rad}_p(\alpha,\beta)$ and it's optimizers which proves Theorem \ref{thmPbcm}. In Section \ref{sectlp}, we characterize all the solutions to linearized problem (\ref{Ppwhl}) and prove Theorem \ref{coroPpwhlp}. Finally, in Section \ref{sect:rtp}, we establish the remainder terms of inequality (\ref{Ppbcm}) in suitable radial space, and we split this section into two subsections in order to deal with the cases $2\leq p<N$ and $1<p<2$ respectively.

\section{{\bfseries Sharp constant and optimizers}}\label{sectpmr}

    Let $\mathcal{D}^{1,p}_{\alpha,r}(\mathbb{R}^N)$ be the radial space of $\mathcal{D}^{1,p}_{\alpha}(\mathbb{R}^N)$ as in (\ref{defd12xr}).
    Then the best constant $S^{rad}_p(\alpha,\beta)$ in (\ref{Ppbcm}) can be defined as
    \begin{equation}\label{defbccknp}
    S^{rad}_p(\alpha,\beta):=\inf_{u\in \mathcal{D}^{1,p}_{\alpha,r}(\mathbb{R}^N)\backslash\{0\}}\frac{\int_{\mathbb{R}^N}|x|^{\alpha}|\nabla u|^p dx}{\left(\int_{\mathbb{R}^N}|x|^{\beta}|u|^{p^*_{\alpha,\beta}} dx\right)^{\frac{p}{p^*_{\alpha,\beta}}}}.
    \end{equation}
    We will use a suitable transform that is changing the variable $r\mapsto r^{\frac{p}{p+\beta-\alpha}}$, related to Sobolev inequality to investigate the sharp constant $S^{rad}_p(\alpha,\beta)$ and optimizers.

    \subsection{Proof of Theorem \ref{thmPbcm}.} We follow the arguments in the proof of \cite[Theorem A.1]{GGN13}. Let $u\in \mathcal{D}^{1,p}_{\alpha,r}(\mathbb{R}^N)$. Making the changes that $v(s)=u(r)$ and $r=s^t$ where $t>0$ will be given later, then we have that
    \begin{equation*}
    \begin{split}
    & \int^\infty_0 r^\alpha|u'(r)|^p r^{N-1}dr
    = t^{1-p}\int^\infty_0|v'(s)|^p s^{(N-p+\alpha)t+p-1}ds.
    \end{split}
    \end{equation*}
    Set
    \begin{equation}\label{defPm}
    K:=(N-p+\alpha)t+p>p,
    \end{equation}
    which implies
    \begin{equation*}
    \begin{split}
    \int^\infty_0 r^\alpha|u'(r)|^p r^{N-1}dr
    = t^{1-p}\int^\infty_0|v'(s)|^p s^{K-1}ds.
    \end{split}
    \end{equation*}

    When $K$ is an integer, we use the classical Sobolev inequality (see \cite{Ta76}) and we get
    \begin{equation*}
    \begin{split}
    \int^\infty_0|v'(s)|^p s^{K-1}ds
    \geq & C_p(K)\left(\int^\infty_0|v(s)|^{\frac{pK}{K-p}}s^{K-1}ds\right)^{\frac{K-p}{K}} \\
    = & t^{-\frac{K-p}{K}}C_p(K)\left(\int^\infty_0|u(r)|^{\frac{pK}{K-p}}r^{\frac{K}{t}-1}dr\right)^{\frac{K-p}{K}},
    \end{split}
    \end{equation*}
    where
    \begin{equation*}
    C_p(K)=\pi^{\frac{p}{2}}K\left(\frac{K-p}{p-1}\right)^{p-1}
    \left(\frac{\Gamma(\frac{K}{p})\Gamma(1+K-K/p)}{\Gamma(K/2+1)\Gamma(K)}\right)^{\frac{p}{K}}
    \left(\frac{\Gamma(K/2)}{2\pi^{K/2}}\right)^{\frac{p}{K}},
    \end{equation*}
    see \cite[(2)]{Ta76}. Moreover, even $K$ is not an integer we readily see that the above inequality remains true.

    In order to get (\ref{Ppbcm}), it requires that $\frac{K}{t}-1=N-1+\beta$, therefore we take
    \begin{equation}\label{pq}
    t=\frac{p}{p+\beta-\alpha},
    \end{equation}
    and thus
    \begin{equation}\label{kpx}
    K=\frac{p(N+\beta)}{p+\beta-\alpha}>p, \quad \frac{pK}{K-p}=\frac{p(N+\beta)}{N-p+\alpha}=p^*_{\alpha,\beta}.
    \end{equation}
    So we get
    \begin{equation*}
    \begin{split}
    \int^\infty_0 r^\alpha|u'(r)|^p r^{N-1}dr
    \geq t^{-p+\frac{p}{K}}C_p(K)\left(\int^\infty_0 r^\beta|u(r)|^{p^*_{\alpha,\beta}}r^{N-1}dr\right)^{\frac{p}{p^*_{\alpha,\beta}}},
    \end{split}
    \end{equation*}
    which proves (\ref{Ppbcm}) with
    \begin{equation*}
    \begin{split}
    S^{rad}_p(\alpha,\beta)
    = & t^{-p+\frac{p}{K}}\omega^{1-\frac{p}{p^*_{\alpha,\beta}}}_{N-1} C_p(K) \\
    = & \left(\frac{p+\beta-\alpha}{p}\right)^{\frac{pN-p+(p-1)\beta+\alpha}{N+\beta}}
    \left(\frac{2\pi^{\frac{N}{2}}}{\Gamma(\frac{N}{2})}\right)^{\frac{p+\beta-\alpha}{N+\beta}}
    C_p\left(\frac{p(N+\beta)}{p+\beta-\alpha}\right),
    \end{split}
    \end{equation*}
    where $\omega_{N-1}$ is the surface area for unit ball of $\mathbb{R}^N$.

    Moreover, from the previous inequalities, we also get that the extremal functions are obtained as
    \begin{equation*}
    \begin{split}
    \int^\infty_0|v'_\nu(s)|^p s^{K-1}ds
    = C_p(K)\left(\int^\infty_0|v_\nu(s)|^{\frac{pK}{K-p}}s^{K-1}ds\right)^{\frac{K-p}{K}}.
    \end{split}
    \end{equation*}
    It is well known that
    \begin{equation*}
    \begin{split}
    v_\nu=A\nu^{\frac{K-p}{p}}\left[1+\nu^{\frac{p}{p-1}}s^{\frac{p}{p-1}}\right]^{-\frac{K-p}{p}}
    \end{split}
    \end{equation*}
    for any $A\in\mathbb{R}$ and $\lambda>0$, see \cite{Au76,Ta76}, or \cite{PV21,Sc16} directly. Setting $\nu=\lambda^{1/t}$ and $s=|x|^{1/t}$, then  we get all the extremal radial functions of $S^{rad}_p(\alpha,\beta)$ have the form
    \begin{equation*}
    W_{\lambda,\alpha,\beta}(x)=\frac{A\lambda^{\frac{N-p+\alpha}{p}}}
    {(1+\lambda^{\frac{p+\beta-\alpha}{p-1}}|x|^{\frac{p+\beta-\alpha}{p-1}})^{\frac{N-p+\alpha}{p+\beta-\alpha}}},
    \end{equation*}
    for any $A\in\mathbb{R}$ and $\lambda>0$. Therefore, the proof of Theorem \ref{thmPbcm} is complete.

    \qed

    Then, let us give a brief statement of Corollary \ref{thmPpwh}. Let $u\in\mathcal{D}^{1,p}_{\alpha,r}(\mathbb{R}^N)$ be a positive radial solution of  (\ref{Ppwh}) .  Making the changes that $v(s)=u(r)$ and $|x|=r=s^t$ where $t=p/(p+\beta-\alpha)$, then by direct calculation,  (\ref{Ppwh}) is equivalent to
    \begin{equation}\label{PpwhlWep}
    \begin{split}
    & -|v'(s)|^{p-2}\left(v''(s)+\frac{K-1}{s}v'(s)+(p-2)\right)
    =t^{-\frac{K-p}{p}}|v|^{\frac{KP}{K-p}-2}v,
    \end{split}
    \end{equation}
    in $s\in (0,+\infty)$, where $K=\frac{p(N+\beta)}{p+\beta-\alpha}>p$, and  $v$ satisfies
    \[
    \int^{+\infty}_0|v'(s)|^p s^{K-1}ds<\infty.
    \]
    Then from \cite{Sc16}, we know $v$ must be the form
    \[
    v(s)=
    \frac{L_{K,t}\nu^{\frac{K-p}{p}}}
    {(1+\nu^{\frac{p}{p-1}}s^{\frac{p}{p-1}})^{\frac{K-p}{p}}}
    \]
    for some $\nu>0$, where
    \[
    L_{K,t}=\left[t^{-p}K\left(\frac{K-p}{p-1}\right)
    ^{p-1}\right]^{\frac{K-p}{p^2}}.
    \]
    That is, equation  (\ref{Ppwh}) has a unique (up to scalings)
    positive radial solution of the form
    \begin{equation*}
    u(x)=\frac{C_{N,\alpha,\beta}\lambda^{\frac{N-p+\alpha}{p}}}
    {(1+\lambda^{\frac{p+\beta-\alpha}{p-1}}|x|^{\frac{p+\beta-\alpha}{p-1}})^{\frac{N-p+\alpha}{p+\beta-\alpha}}},
    \end{equation*}
    where $\lambda=\nu^{\frac{p}{p+\beta-\alpha}}$, and
    \[
    C_{N,\alpha,\beta}=\left[(N+\beta)\left(\frac{N-p+\alpha}{p-1}\right)^{p-1}\right]
    ^{\frac{N-p+\alpha}{p(p+\beta-\alpha)}}.
    \]
    Thus, Corollary \ref{thmPpwh} holds.

\section{{\bfseries Linearized problem}}\label{sectlp}

    For simplicity of notations, we write $U$ instead of $U_{1,\alpha,\beta}$ if there is no possibility of confusion. First of all, let us rewrite the linear equation (\ref{Ppwhl}) as
    \begin{align}\label{rwlp}
    & -|x|^2\Delta W -\alpha (x\cdot\nabla W)
    -(p-2) \sum^{N}_{i,j=1}\frac{\partial^2 W}{\partial x_i\partial x_j}x_i x_j
    -\frac{(p-2)(N+\beta)}{1+|x|^{\frac{p+\beta-\alpha}{p-1}}} (x\cdot\nabla W) \nonumber\\
    = & (p^*_{\alpha,\beta}-1)C_{N,\alpha,\beta}^{p^*_{\alpha,\beta}-p}\left(\frac{N-p+\alpha}{p-1}\right)^{2-p}
    \frac{|x|^{\frac{p+\beta-\alpha}{p-1}}}{(1+|x|^{\frac{p+\beta-\alpha}{p-1}})^2}W\quad \mbox{in}\quad \mathbb{R}^N,\quad W\in L^2_{\beta,*}(\mathbb{R}^N).
    \end{align}
    Indeed a straightforward computation shows that
    \begin{align}\label{rwlpy}
    & {\rm div}(|x|^{\alpha}|\nabla U|^{p-2}\nabla W)+(p-2){\rm div}(|x|^{\alpha}|\nabla U|^{p-4}(\nabla U\cdot\nabla W)\nabla U) \nonumber\\
    = & |x|^{\alpha}|\nabla U|^{p-2}\Delta W + \nabla(|x|^{\alpha}|\nabla U|^{p-2})\cdot\nabla W \nonumber\\
     & + (p-2)|x|^{\alpha}|\nabla U|^{p-4}(\nabla U\cdot\nabla W)\Delta U \nonumber\\
     & + (p-2)(\nabla U\cdot\nabla W) (\nabla (|x|^{\alpha}|\nabla U|)\cdot\nabla U) \nonumber\\
     & + (p-2)|x|^{\alpha}|\nabla U| (\nabla (\nabla U \cdot \nabla W)\cdot\nabla U) \nonumber\\
    = & |x|^{\alpha}\Big\{|\nabla U|^{p-2}\left[\Delta W+\alpha|x|^{-2}(x\cdot\nabla W)\right]\nonumber \\
      & + (p-2)|\nabla U|^{p-4}\big[ (\nabla U\cdot\nabla W)\Delta U +2(\nabla U\nabla (\nabla U)\cdot\nabla W)
       \nonumber\\
      & + \alpha |x|^{-2}(\nabla U\cdot\nabla W)(x\cdot\nabla U)
      + (\nabla U\nabla(\nabla W) \cdot \nabla U)\big] \nonumber\\
      & + (p-2)(p-4)|\nabla U|^{p-6}(\nabla U\cdot\nabla W) (\nabla U \nabla(\nabla U) \cdot\nabla U)\Big\},
    \end{align}
    and
    \begin{align}\label{rwlpnu}
    \nabla U
    = & -\frac{c_{N,p}|x|^{\frac{2-p+\beta-\alpha}{p-1}}x}
    {(1+|x|^{\frac{p+\beta-\alpha}{p-1}})^{\frac{N+\beta}{p+\beta-\alpha}}}  \nonumber\\
    (x\cdot\nabla U)
    = &  -\frac{c_{N,p}|x|^{\frac{p+\beta-\alpha}{p-1}}}
    {(1+|x|^{\frac{p+\beta-\alpha}{p-1}})^{\frac{N+\beta}{p+\beta-\alpha}}} \nonumber\\
    (\nabla U\cdot\nabla W)
    = & -\frac{c_{N,p}|x|^{\frac{2-p+\beta-\alpha}{p-1}}}
    {(1+|x|^{\frac{p+\beta-\alpha}{p-1}})^{\frac{N+\beta}{p+\beta-\alpha}}}(x\cdot\nabla W),
    \end{align}
    furthermore,
    \begin{align}\label{rwlpnup2}
    \Delta U
    = & \frac{-c_{N,p}}
    {(1+|x|^{\frac{p+\beta-\alpha}{p-1}})^{\frac{N+\beta}{p+\beta-\alpha}}}
    \left\{\left(\frac{2-p+\beta-\alpha}{p-1}+N\right)|x|^{\frac{2-p+\beta-\alpha}{p-1}}
    -\frac{\frac{N+\beta}{p-1}|x|^{\frac{2+2\beta-2\alpha}{p-1}}}{1+|x|^{\frac{p+\beta-\alpha}{p-1}}}
    \right\} \nonumber\\
    \sum^N_{j=1}\frac{\partial U}{\partial x_j}\frac{\partial^2 U}{\partial x_i\partial x_j}
    = & \frac{c^2_{N,p}}
    {(1+|x|^{\frac{p+\beta-\alpha}{p-1}})^{\frac{2(N+\beta)}{p+\beta-\alpha}}}
    \left\{\frac{1+\beta-\alpha}{p-1}|x|^{\frac{2(2-p+\beta-\alpha)}{p-1}}
    -\frac{\frac{N+\beta}{p-1}|x|^{\frac{4-p+3\beta-3\alpha}{p-1}}}{1+|x|^{\frac{p+\beta-\alpha}{p-1}}}
    \right\}x_i,
    \end{align}
    where
    \begin{equation*}
    \begin{split}
    c_{N,p}:=C_{N,\alpha,\beta}\frac{N-p+\alpha}{p-1}.
    \end{split}
    \end{equation*}
    Here,
    \begin{align*}
    (\nabla U\nabla (\nabla U)\cdot\nabla W)
    = & \sum^N_{i,j}\frac{\partial U}{\partial x_j}\frac{\partial^2 U}{\partial x_i\partial x_j}\frac{\partial W}{\partial x_i} \\
    = & \frac{c^2_{N,p}(x\cdot\nabla W)}
    {(1+|x|^{\frac{p+\beta-\alpha}{p-1}})^{\frac{2(N+\beta)}{p+\beta-\alpha}}}
    \left\{\frac{1+\beta-\alpha}{p-1}|x|^{\frac{2(2-p+\beta-\alpha)}{p-1}}
    -\frac{\frac{N+\beta}{p-1}|x|^{\frac{4-p+3\beta-3\alpha}{p-1}}}{1+|x|^{\frac{p+\beta-\alpha}{p-1}}}
    \right\} \\
    (\nabla U\nabla (\nabla W)\cdot\nabla U)
    = & \sum^N_{i,j}\frac{\partial U}{\partial x_j}\frac{\partial^2 W}{\partial x_i\partial x_j}\frac{\partial U}{\partial x_i} \\
    = & \frac{c^2_{N,p}|x|^{\frac{2(2-p+\beta-\alpha)}{p-1}}}
    {(1+|x|^{\frac{p+\beta-\alpha}{p-1}})^{\frac{2(N+\beta)}{p+\beta-\alpha}}}
    \sum^N_{i,j}\frac{\partial^2 W}{\partial x_i\partial x_j}x_ix_j  \\
    (\nabla U\nabla(\nabla U) \cdot\nabla U)
    = & \sum^N_{i,j}\frac{\partial U}{\partial x_j}\frac{\partial^2 U}{\partial x_i\partial x_j}\frac{\partial U}{\partial x_i} \\
    = & \frac{-c^3_{N,p}|x|^{\frac{p+\beta-\alpha}{p-1}}}
    {(1+|x|^{\frac{p+\beta-\alpha}{p-1}})^{\frac{3(N+\beta)}{p+\beta-\alpha}}}
    \left\{\frac{1+\beta-\alpha}{p-1}|x|^{\frac{2(2-p+\beta-\alpha)}{p-1}}
    -\frac{\frac{N+\beta}{p-1}|x|^{\frac{4-p+3\beta-3\alpha}{p-1}}}{1+|x|^{\frac{p+\beta-\alpha}{p-1}}}
    \right\}.
    \end{align*}

    Then by using the standard spherical decomposition and taking the changes of variable $r\mapsto r^{\frac{p}{p+\beta-\alpha}}$, we can characterize all solutions to the linearized problem (\ref{rwlp}).

    \subsection{Proof of Theorem \ref{coroPpwhlp}.} Since $U$ is radial we can make a partial wave decomposition of (\ref{rwlp}), namely
    \begin{equation}\label{Ppwhl2defvdp}
    W(r,\theta)=\sum^{\infty}_{k=0}\varphi_k(r)\Psi_k(\theta),\quad \mbox{where}\quad r=|x|,\quad \theta=\frac{x}{|x|}\in \mathbb{S}^{N-1},
    \end{equation}
    and
    \begin{equation*}
    \varphi_k(r)=\int_{\mathbb{S}^{N-1}}W(r,\theta)\Psi_k(\theta)d\theta.
    \end{equation*}
    Here $\Psi_k(\theta)$ denotes the $k$-th spherical harmonic, i.e., it satisfies
    \begin{equation}\label{deflk}
    -\Delta_{\mathbb{S}^{N-1}}\Psi_k=\lambda_k \Psi_k,
    \end{equation}
    where $\Delta_{\mathbb{S}^{N-1}}$ is the Laplace-Beltrami operator on $\mathbb{S}^{N-1}$ with the standard metric and  $\lambda_k$ is the $k$-th eigenvalue of $-\Delta_{\mathbb{S}^{N-1}}$. It is well known that \begin{equation}\label{deflklk}
    \lambda_k=k(N-2+k),\quad k=0,1,2,\ldots,
    \end{equation}
    whose multiplicity is $\frac{(N+2k-2)(N+k-3)!}{(N-2)!k!}$ and that \[{\rm Ker}(\Delta_{\mathbb{S}^{N-1}}+\lambda_k)=\mathbb{Y}_k(\mathbb{R}^N)|_{\mathbb{S}^{N-1}},\] where $\mathbb{Y}_k(\mathbb{R}^N)$ is the space of all homogeneous harmonic polynomials of degree $k$ in $\mathbb{R}^N$. It is standard that $\lambda_0=0$ and the corresponding eigenfunction of (\ref{deflk}) is the constant function. The second eigenvalue $\lambda_1=N-1$ and the corresponding eigenfunctions of (\ref{deflk}) are $x_i/|x|$, $i=1,\ldots,N$.

    The following results can be obtained by direct calculation,
    \begin{align}\label{Ppwhl2deflklwp}
    \Delta (\varphi_k(r)\Psi_k(\theta))
    = & \Psi_k\left(\varphi''_k+\frac{N-1}{r}\varphi'_k\right)+\frac{\varphi_k}{r^2}\Delta_{\mathbb{S}^{N-1}}\Psi_k \nonumber\\
    = & \Psi_k\left(\varphi''_k+\frac{N-1}{r}\varphi'_k-\frac{\lambda_k}{r^2}\varphi_k\right).
    \end{align}
    It is easy to verify that
    \begin{equation*}
    \frac{\partial (\varphi_k(r)\Psi_k(\theta))}{\partial x_i}=\varphi'_k\frac{x_i}{r}\Psi_k+\varphi_k\frac{\partial\Psi_k}{\partial \theta_l}\frac{\partial\theta_l}{\partial x_i},
    \end{equation*}
    hence
    \begin{equation}\label{Ppwhl2deflklnp}
    \begin{split}
    x\cdot\nabla (\varphi_k(r)\Psi_k(\theta))=\sum^{N}_{i=1}x_i\frac{\partial (\varphi_k(r)\Psi_k(\theta))}{\partial x_i}=\varphi'_kr\Psi_k+\varphi_k\frac{\partial\Psi_k}{\partial \theta_l}\sum^{N}_{i=i}\frac{\partial\theta_l}{\partial x_i}x_i=\varphi'_kr\Psi_k,
    \end{split}
    \end{equation}
    and
    \begin{align}\label{npp2}
    \sum^N_{i,j=1}\frac{\partial^2 (\varphi_k(r)\Psi_k(\theta))}{\partial x_i\partial x_j}x_ix_j
    = & 2 \varphi'_kr\frac{\partial\Psi_k}{\partial \theta_l}\sum^N_{i=1}\frac{\partial\theta_l}{\partial x_i}x_i
    + \varphi_k\frac{\partial^2\Psi_k}{\partial \theta_l\partial \theta_m}\sum^N_{i,j=1}\frac{\partial\theta_l}{\partial x_i}x_i\frac{\partial\theta_m}{\partial x_j}x_j\nonumber \\
    & + \frac{\partial\Psi_k}{\partial \theta_l}\varphi_k\sum^N_{i,j=1}\frac{\partial^2\theta_l}{\partial x_i\partial x_j}x_ix_j
    +\varphi''_kr^2\Psi_k
    = \varphi''_kr^2\Psi_k,
    \end{align}
    since
    \begin{equation*}
    \begin{split}
    \sum^N_{i=1}\frac{\partial\theta_l}{\partial x_i}x_i=0\quad \mbox{and}\quad \sum^N_{i,j=1}\frac{\partial^2\theta_l}{\partial x_i\partial x_j}x_ix_j=0,\quad l=1,\ldots,N-1.
    \end{split}
    \end{equation*}
    Then putting together (\ref{Ppwhl2defvdp}), (\ref{Ppwhl2deflklwp}), (\ref{Ppwhl2deflklnp}) and (\ref{npp2}) into (\ref{rwlp}), the function $W^{(p)}$ is a solution of (\ref{rwlp}) if and only if $\varphi_k\in\mathcal{W}$ is a classical solution of the system
    \begin{eqnarray}\label{Ppwhl2p2tpy}
    \left\{ \arraycolsep=1.5pt
       \begin{array}{ll}
        (p-1)\varphi''_k+\frac{\varphi'_k}{r}\left[(N+\alpha-1)+\frac{(p-2)(N+\beta)}{1+r^{\frac{p+\beta-\alpha}{p-1}}}\right]
    -\frac{\lambda_k}{r^2}\varphi_k \\[4mm]
    + (p^*_{\alpha,\beta}-1)C_{N,\alpha,\beta}^{p^*_{\alpha,\beta}-p}\left(\frac{N-p+\alpha}{p-1}\right)^{2-p}
    \frac{r^{\frac{p+\beta-\alpha}{p-1}-2}}{\left(1+r^{\frac{p+\beta-\alpha}{p-1}}\right)^2}\varphi_k=0 \quad \mbox{in}\quad r\in(0,\infty),\\[4mm]
        \varphi'_k(0)=0 \quad\mbox{if}\quad k=0,\quad \mbox{and}\quad \varphi_k(0)=0 \quad\mbox{if}\quad k\geq 1,
        \end{array}
    \right.
    \end{eqnarray}
    where $\mathcal{W}:=\{\omega\in C^1([0,\infty))| \int^\infty_0 r^\beta U^{p^*_{
    \alpha,\beta}-2}|\omega(r)|^2 r^{N-1} dr<\infty\}$.
    We use the change of variables $|x|=r=s^t$ where $t=p/(p+\beta-\alpha)$ and let
    \begin{equation}\label{Ppwhl2p2txyp}
    \eta_k(s)=\varphi_k(r),
    \end{equation}
    that transforms (\ref{Ppwhl2p2tpy}) into the following equations for any $\eta_k\in \widetilde{\mathcal{W}}$, $k=0,1,2,\ldots,$
    \begin{equation}\label{Ppwhl2p2tp}
    \eta''_k+\frac{\eta'_k}{s}\left(\frac{K-1}{p-1}+\frac{(p-2)K}{(p-1)(1+s^{\frac{p}{p-1}})}\right)-\frac{t^2\lambda_k}{(p-1)s^2}\eta_k
        +\frac{K(Kp-K+p)}{(p-1)^2}\frac{s^{\frac{p}{p-1}-2}}{(1+s^{\frac{p}{p-1}})^2}\eta_k=0.
    \end{equation}
    where $\widetilde{\mathcal{W}}:=\{\omega\in C^1([0,\infty))| \int^\infty_0 V^{\frac{Kp}{K-p}-2}|\omega(s)|^2 s^{K-1} ds<\infty\}$, here $V(s)=U(r)$ and
    \begin{equation}
    K=\frac{p(N+\beta)}{p+\beta-\alpha}>p.
    \end{equation}
    Here we have used the fact  
    \begin{equation*}
    t^2(p^*_{\alpha,\beta}-1)C_{N,\alpha,\beta}^{p^*_{\alpha,\beta}-p}\left(\frac{N-p+\alpha}{p-1}\right)^{2-p}
    =\frac{K(Kp-K+p)}{p-1}.
    \end{equation*}

    Fix $K$ let us now consider the following eigenvalue problem
    \begin{equation}\label{Ppwhl2p2tep}
    \eta''+\frac{\eta'}{s}\left(\frac{K-1}{p-1}+\frac{(p-2)K}{(p-1)(1+s^{\frac{p}{p-1}})}\right)-\frac{\mu}{(p-1)s^2}\eta
        +\frac{K(Kp-K+p)}{(p-1)^2}\frac{s^{\frac{p}{p-1}-2}}{(1+s^{\frac{p}{p-1}})^2}\eta=0.
    \end{equation}
    When $K$ is an integer we can study (\ref{Ppwhl2p2tep}) as the linearized operator of the equation
    \begin{equation*}
    -{\rm div}(|\nabla u|^{p-2}\nabla u)=K\left(\frac{K-p}{p-1}\right)^{p-1}u^{\frac{Kp}{K-p}-1} ,\quad u>0 \quad \mbox{in}\quad \mathbb{R}^N, \quad u\in \mathcal{D}^{1,p}_0(\mathbb{R}^K).
    \end{equation*}
    around the standard solution $u(x)=(1+|x|^{p/(p-1)})^{-(K-p)/p}$ (note that we always have $K>p$). In this case, as in \cite[Proposition 3.1]{FN19}, we have that
    \begin{equation}\label{Ppwhl2ptevp}
    \mu_0=0; \quad \mu_1=K-1\quad \mbox{and}\quad \eta_0(s)=\frac{(p-1)-s^{\frac{p}{p-1}}}{(1+s^{\frac{p}{p-1}})^\frac{K}{p}}; \quad \eta_1(s)=\frac{s^{\frac{1}{p-1}}}{(1+s^{\frac{p}{p-1}})^\frac{K}{p}}.
    \end{equation}
    Moreover, even $K$ is not an integer we readily see that (\ref{Ppwhl2ptevp}) remains true. Therefore, we can conclude that (\ref{Ppwhl2p2txyp}) has nontrivial solutions if and only if
    \begin{equation*}
    t^2\lambda_k\in \{0,K-1\},
    \end{equation*}
    where $\lambda_k$ is given in \eqref{deflklk}. If $t^2\lambda_k=0$ then $k=0$. Moreover, if $t^2\lambda_k=K-1$, that is
    \begin{equation}\label{ppk}
    \frac{(p+\beta-\alpha)[(N-1)p+\beta(p-1)+\alpha]}{p^2}=k(N-2+k)\quad \mbox{for some}\quad k\in\mathbb{N}^+.
    \end{equation}
   Turning back to (\ref{Ppwhl2p2tpy}) we obtain the solutions that if (\ref{ppk}) holds,
    \begin{equation}\label{Ppwhl2pyep}
    \varphi_0(r)=\frac{(p-1)-r^{\frac{p+\beta-\alpha}{p-1}}}{(1+r^{\frac{p+\beta-\alpha}{p-1}})^\frac{N+\beta}{p+\beta-\alpha}},\quad
    \varphi_k(r)=\frac{r^{\frac{p+\beta-\alpha}{p(p-1)}}}{(1+r^{\frac{p+\beta-\alpha}{p-1}})^\frac{N+\beta}{p+\beta-\alpha}},
    \end{equation}
    otherwise
    \begin{equation}\label{Ppwhl2pyfp}
    \varphi_0(r)=\frac{(p-1)-r^{\frac{p+\beta-\alpha}{p-1}}}{(1+r^{\frac{p+\beta-\alpha}{p-1}})^\frac{N+\beta}{p+\beta-\alpha}}.
    \end{equation}
    That is, if (\ref{ppk}) holds then the space of solutions of (\ref{rwlp}) has dimension $1+\frac{(N+2k-2)(N+k-3)!}{(N-2)!k!}$ and is spanned by
    \begin{equation*}
    W_0(x)=\frac{(p-1)-|x|^{\frac{p+\beta-\alpha}{p-1}}}{(1+|x|^{\frac{p+\beta-\alpha}{p-1}})^\frac{N+\beta}{p+\beta-\alpha}},\quad W_{k,i}(x)=\frac{|x|^{\frac{p+\beta-\alpha}{p(p-1)}}\Psi_{k,i}(x)}{(1+|x|^{\frac{p+\beta-\alpha}{p-1}})^\frac{N+\beta}{p+\beta-\alpha}},
    \end{equation*}
    where $\{\Psi_{k,i}\}$, $i=1,\ldots,\frac{(N+2k-2)(N+k-3)!}{(N-2)!k!}$, form a basis of $\mathbb{Y}_k(\mathbb{R}^N)$, the space of all homogeneous harmonic polynomials of degree $k$ in $\mathbb{R}^N$.
    Otherwise the space of solutions of (\ref{rwlp}) has dimension one and is spanned by $W_0(x)$.
    Thus, the proof of Theorem \ref{coroPpwhlp} is complete.

    \qed

\section{{\bfseries Remainder terms of (CKN) inequality (\ref{Ppbcm})}}\label{sect:rtp}

    In this section, we follow the steps of Figalli and Zhang \cite{FZ22} to give the gradient remainder term of (CKN) inequality (\ref{Ppbcm}) in radial space for all $1<p<N$.

    Firstly, we introduce some useful inequalities.

    \begin{lemma}\label{lemui1p}
    \cite[Lemmas 2.1]{FZ22} Let $x, y\in\mathbb{R}^N$. Then for any $\kappa>0$, there exists a constant $\mathcal{C}_1=\mathcal{C}_1(p,\kappa)>0$ such that
     the following inequalities hold.

    $\bullet$ For $1<p<2$,
    \begin{align}\label{uinb1p}
    |x+y|^p
    \geq & |x|^p+ p|x|^{p-2}x\cdot y+ \frac{1-\kappa}{2}\left(p|x|^{p-2}|y|^2+ p(p-2)|\omega|^{p-2}(|x|-|x+y|)^2 \right) \nonumber\\
    & +\mathcal{C}_1\min\{|y|^p,|x|^{p-2}|y|^2\},
    \end{align}
    where
    \begin{eqnarray*}
    \omega=\omega(x,x+y)=
    \left\{ \arraycolsep=1.5pt
       \begin{array}{ll}
        \left(\frac{|x+y|}{(2-p)|x+y|+(p-1)|x|}\right)^{\frac{1}{p-2}}x,\ \ &{\rm if}\ \ |x|<|x+y|,\\[3mm]
        x,\ \ &{\rm if}\ \  |x+y|\leq |x|.
        \end{array}
    \right.
    \end{eqnarray*}
    Furthermore, it is easy to verify that $|x|^{p-2}|y|^2+(p-2)|\omega|^{p-2}(|x|-|x+y|)^2\geq 0$.

    $\bullet$ For $p\geq 2$,
    \begin{align}\label{uinb2p}
    |x+y|^p
    \geq & |x|^p+ p|x|^{p-2}x\cdot y+ \frac{1-\kappa}{2}\left(p|x|^{p-2}|y|^2+ p(p-2)|\omega|^{p-2}(|x|-|x+y|)^2 \right) \nonumber\\
    & +\mathcal{C}_1 |y|^p ,
    \end{align}
    where
    \begin{eqnarray*}
    \omega=\omega(x,x+y)=
    \left\{ \arraycolsep=1.5pt
       \begin{array}{ll}
        x,\ \ &{\rm if}\ \ |x|<|x+y|,\\[3mm]
        \left(\frac{|x+y|}{|x|}\right)^{\frac{1}{p-2}}(x+y),\ \ &{\rm if}\ \  |x+y|\leq |x|.
        \end{array}
    \right.
    \end{eqnarray*}
    \end{lemma}

    \begin{lemma}\label{lemui1p*l}
    Let $a, b\in\mathbb{R}$. Then for any $\kappa>0$, there exists a constant $\mathcal{C}_2=\mathcal{C}_2(p^*_{\alpha,\beta},\kappa)>0$ where $p^*_{\alpha,\beta}=\frac{p(N+\beta)}{N-p+\alpha}$ such that
     the following inequalities hold.

    $\bullet$ For $1<p\leq\frac{2(N+\alpha)}{N+2+\beta}$,
    \begin{equation}\label{uinx2pl}
    |a+b|^{p^*_{\alpha,\beta}}
    \leq |a|^{p^*_{\alpha,\beta}}+ p^*_{\alpha,\beta}|a|^{p^*_{\alpha,\beta}-2}a b
    + \left(\frac{p^*_{\alpha,\beta}(p^*_{\alpha,\beta}-1)}{2}+\kappa\right)\frac{(|a|+\mathcal{C}_2|b|)^{{p^*_{\alpha,\beta}}}|b|^2}{|a|^2+|b|^2}|b|^{2}.
    \end{equation}

    $\bullet$ For $\frac{2(N+\alpha)}{N+2+\beta}< p<N$,
    \begin{equation}\label{uinx2pb}
    |a+b|^{p^*_{\alpha,\beta}}
    \leq |a|^{p^*_{\alpha,\beta}}+ p^*_{\alpha,\beta}|a|^{p^*_{\alpha,\beta}-2}a b
    + \left(\frac{p^*_{\alpha,\beta}(p^*_{\alpha,\beta}-1)}{2}+\kappa\right)|a|^{{p^*_{\alpha,\beta}}-2}|b|^2 +\mathcal{C}_2|b|^{p^*_{\alpha,\beta}}.
    \end{equation}
    \end{lemma}

    \begin{proof}
    We notice that $1<p\leq\frac{2(N+\alpha)}{N+2+\beta}$ indicates $p^*_{\alpha,\beta}\leq2$ and $\frac{2(N+\alpha)}{N+2+\beta}< p<N$ indicates $p^*_{\alpha,\beta}> 2$, then the above inequalities directly follow from \cite[Lemma 3.2]{FN19} and \cite[Lemma 2.4]{FZ22}.
    \end{proof}

    For simplicity of notations, we write $U_{\lambda}$ instead of $U_{\lambda,\alpha,\beta}$  as in (\ref{defeu}) and $S_r$ instead of $S^{rad}_p(\alpha,\beta)$, particularly, we write $U=U_1$ instead of $U_{1,\alpha,\beta}$
    if there is no possibility of confusion. Moreover, in order to shorten formulas, for each $u\in \mathcal{D}^{1,p}_{\alpha}(\mathbb{R}^N)$ we denote
    \begin{equation*}
    \begin{split}
    \|u\|:
    =\left(\int_{\mathbb{R}^N}|x|^{\alpha}|\nabla u|^p dx\right)^{\frac{1}{p}},
    \quad \|u\|_*: =\left(\int_{\mathbb{R}^N}|x|^{\beta}|u|^{p^*_{\alpha,\beta}} dx\right)^{\frac{1}{p^*_{\alpha,\beta}}}.
    \end{split}
    \end{equation*}

    Denote $C^1_{c,0}(\mathbb{R}^N)$  be the space of compactly supported functions of class $C^1$ that are constant in a neighborhood of the origin, then we define the weighted Sobolev space $\mathcal{D}^{1,2}_{\alpha,*}(\mathbb{R}^N)$ as the completion of $C^1_{c,0}(\mathbb{R}^N)$ with respect to the
    inner product
    \[
    \langle u,v\rangle_*=\int_{\mathbb{R}^N}|x|^\alpha|\nabla U|^{p-2} \nabla u\cdot \nabla v dx,
    \]
    and the norm
    \begin{equation}\label{defd12*nb}
    \|u\|_{\mathcal{D}^{1,2}_{\alpha,*}(\mathbb{R}^N)}:=\langle u,u\rangle_*^{\frac{1}{2}}=\left(\int_{\mathbb{R}^N}|x|^\alpha|\nabla U|^{p-2} |\nabla u|^2  dx\right)^{\frac{1}{2}}.
    \end{equation}

    \begin{remark}\label{remdefsn2}\rm
    As stated in \cite[Remark 3.1]{FZ22}, it is important for us to considwer weighted that are not necessarily integrable at the origin, since $|x|^{\alpha}|\nabla U|^{p-2}\sim |x|^{\frac{(p-2)(1+\beta-\alpha)}{p-1}+\alpha}\not\in L^1(B_1)$ for $p\leq \frac{N+\alpha +2(1+\beta-\alpha)}{N+1+\beta}$. This is why, when defining weighted Sobolev spaces, we consider the space $C^1_{c,0}(\mathbb{R}^N)$, so that gradients vanish near zero. Of course, replacing $C^1(\mathbb{R}^N)$ by $C^1_{c,0}(\mathbb{R}^N)$ plays no role in the case $p> \frac{N+\alpha +2(1+\beta-\alpha)}{N+1+\beta}$.
    \end{remark}

    The following embedding theorem generalizes \cite[Corollary 6.2]{FN19}.
    \begin{proposition}\label{propcet}
    Let $1<p<N$. The space $\mathcal{D}^{1,2}_{\alpha,*}(\mathbb{R}^N)$ compactly embeds into $L^2_{\beta,*}(\mathbb{R}^N)$ at least for radial case, where $L^2_{\beta,*}(\mathbb{R}^N)$ is defined in \eqref{defd12*n}.
    \end{proposition}

    \begin{proof}
    Let us recall that \cite[Proposition 2.2]{PV21} indicates $\mathcal{D}^{1,2}_{0,*}(\mathbb{R}^N)$ continuously embeds into $L^2_{0,*}(\mathbb{R}^N)$, furthermore, \cite[Corollary 6.2]{FN19} indicates $\mathcal{D}^{1,2}_{0,*}(\mathbb{R}^N)$ compactly embeds into $L^2_{0,*}(\mathbb{R}^N)$ for $2\leq p<N$, indeed, it holds also for all $1<p<N$, see \cite[Proposition 3.2]{FZ22}.

    Then let $\{u_n\}\subset \mathcal{D}^{1,2}_{\alpha,*}(\mathbb{R}^N)$ be radial, satisfying $u_n\rightharpoonup u$ in $\mathcal{D}^{1,2}_{\alpha,*}(\mathbb{R}^N)$ for some raidal $u\in\mathcal{D}^{1,2}_{\alpha,*}(\mathbb{R}^N)$, the conclusion follows from $u_n\to u$ in $L^2_{\beta,*}(\mathbb{R}^N)$, that is, we need to prove
    \begin{equation}\label{pevpce}
    \int_{\mathbb{R}^N}|x|^\beta U^{p^*_{\alpha,\beta}-2}(u_n^2-u^2) dx\to 0.
    \end{equation}
    Let $u_n(r)=v_n(s)$, where $r=s^{t}$ with $t=\frac{p}{p+\beta-\alpha}$, then
    \begin{equation*}
    \begin{split}
    \int_{\mathbb{R}^N}|x|^\beta U^{p^*_{\alpha,\beta}-2}(u_n^2-u^2) dx
    = & \omega_{N-1}t\int^\infty_0 V^{\frac{Kp}{K-p}-2}(v^2_n(s)-v^2(s))s^{K-1}ds,
    \end{split}
    \end{equation*}
    where $K=\frac{p(N+\beta)}{p+\beta-\alpha}>p$ and $V(s)=U(r)$.
    From \cite[Corollary 6.2]{FN19} for the radial case, we deduce that
    \begin{equation*}
    \begin{split}
    \int^\infty_0 V^{\frac{Kp}{K-p}-2}(v^2_n(s)-v^2(s))s^{K-1}ds\to 0,
    \end{split}
    \end{equation*}
    thus \eqref{pevpce} holds.
    \end{proof}

    Let us consider the following eigenvalue problem
    \begin{equation}\label{pevp}
    \begin{split}
    & \mathcal{L}_{U} [v]=\mu|x|^{\beta} U^{p^*_{\alpha,\beta}-2}v \quad \mbox{in}\quad \mathbb{R}^N,\quad v\in L^2_{\beta,*}(\mathbb{R}^N),
    \end{split}
    \end{equation}
    where
    \begin{equation*}
    \begin{split}
    & \mathcal{L}_{U} [v]:=-{\rm div}(|x|^{\alpha}|\nabla U|^{p-2}\nabla v)-(p-2){\rm div}(|x|^{\alpha}|\nabla U|^{p-4}(\nabla U\cdot\nabla v)\nabla U).
    \end{split}
    \end{equation*}
    Let $\mathcal{M}:=\{cU_{\lambda}: c\in\mathbb{R}, \lambda>0\}$ be set of optimizers for (CKN) inequality (\ref{Ppbcm}). Then according to Theorem \ref{coroPpwhlp} we have

    \begin{theorem}\label{propev}
    Let $\mu_i$, $i=1,2,\ldots,$ denote the eigenvalues of (\ref{pevp}) in increasing order. Then $\mu_1=(p-1)$ is simple and the corresponding eigenfunction is $\zeta U$ with $\zeta\in\mathbb{R}$. If
    \begin{equation}\label{npkev}
    \left(\frac{p+\beta-\alpha}{p}\right)^2\left[\frac{p(N+\beta)}{p+\beta-\alpha}-1\right]=k(N-2+k), \quad \mbox{for some}\quad k\in\mathbb{N}^+,
    \end{equation}
    then $\mu_{2}=\mu_{3}=\cdots=\mu_{M_k+2}=p^*_{\alpha,\beta}-1$, where $M_k:=\frac{(N+2k-2)(N+k-3)!}{(N-2)!k!}$, with the corresponding $(1+M_k)$-dimensional eigenfunction space $T_{U} \mathcal{M}$ spanned by
    \[
    \left\{W_0,\quad W_{k,i}, i=1,\ldots, M_k\right\},
    \]
    where $W_0$ and $W_{k,i}$ are given as in (\ref{defwp0i}), and the Rayleigh quotient characterization of eigenvalues implies
    \begin{equation*}
    \mu_{M_k+3}=\inf \left\{
    \frac{\int_{\mathbb{R}^N}\mathcal{L}_{U} [v] v dx}{\int_{\mathbb{R}^N}|x|^{\beta} U^{p^*_{\alpha,\beta}-2}v^2 dx}:\quad v\perp {\rm Span}\{W_0,\quad W_{k,i}, i=1,\ldots, M_k\}
    \right\}>p^*_{\alpha,\beta}-1.
    \end{equation*}
    Otherwise, that is, if (\ref{npkev}) does not hold, then $\mu_2=p^*_{\alpha,\beta}-1$ with the corresponding one-dimensional eigenfunction space $T_{U} \mathcal{M}$ spanned by $W_0$, and $p^*_{\alpha,\beta}-1<\mu_{3}$.
    \end{theorem}

    From Theorem \ref{propev}, we directly obtain
    \begin{proposition}\label{propevl}
    Let $1< p<N$. There exists a constant $\tau=\tau(N,p,\alpha,\beta)>0$ such that for any function $v\in L^2_{\beta,*}(\mathbb{R}^N)$ orthogonal to $T_{U} \mathcal{M}$, it holds that
    \begin{equation*}
    \begin{split}
    & \int_{\mathbb{R}^N}|x|^{\alpha}\left[|\nabla U|^{p-2}|\nabla v|^2+(p-2)|\nabla U|^{p-4}|\nabla U\cdot\nabla v|^2\right]dx \\
    \geq & \left[(p^*_{\alpha,\beta}-1)+2\tau\right]
    \int_{\mathbb{R}^N}|x|^{\beta}U^{p^*_{\alpha,\beta}-2}|v|^2dx.
    \end{split}
    \end{equation*}
    \end{proposition}

    Following \cite{FZ22}, we give the following remark which will  be important to give a meaning to the notion of  ``orthogonal to $T_{U} \mathcal{M}$" for functions which are not necessarily in $L^2_{\beta,*}(\mathbb{R}^N)$.

    \begin{remark}\rm
    For any $\xi\in T_{U} \mathcal{M}$ it holds $U^{p^*_{\alpha,\beta}-2}\xi \in L_\beta^{\frac{p^*_{\alpha,\beta}}{p^*_{\alpha,\beta}-1}}(\mathbb{R}^N)
    =\left(L_\beta^{p^*_{\alpha,\beta}}(\mathbb{R}^N)\right)'$, here $L_\beta^{q}(\mathbb{R}^N)$ is the set of measurable functions with the norm $\|\varphi\|_{L_\beta^{q}(\mathbb{R}^N)}:
    =\left(\int_{\mathbb{R}^N}|x|^{\beta}|\varphi|^{q} dx\right)^{\frac{1}{q}}$.
    Hence, by abuse of notation, for any function $v\in L_\beta^{p^*_{\alpha,\beta}}(\mathbb{R}^N)$ we say that $v$ is orthogonal to $T_{U} \mathcal{M}$ in $L^2_{\beta,*}(\mathbb{R}^N)$ if \[
    \int_{\mathbb{R}^N}|x|^{\beta}U^{p^*_{\alpha,\beta}-2}\xi vdx=0,\quad \forall \xi \in T_{U} \mathcal{M}.
    \]
    \end{remark}

    Note that, by H\"{o}lder inequality, $L_\beta^{p^*_{\alpha,\beta}}(\mathbb{R}^N)\subset L^2_{\beta,*}(\mathbb{R}^N)$ if $p^*_{\alpha,\beta}\geq 2$.
    Hence, the notion of orthogonality introduced above is particularly relevant when $p^*_{\alpha,\beta}<2$ (equivalently, $p<\frac{2(N+\beta)}{N+2+\alpha}$). We also observe that, by Sobolev embedding, the previous remark gives a meaning to the orthogonality to $T_{U} \mathcal{M}$ for functions in $\mathcal{D}^{1,p}_{\alpha,r}(\mathbb{R}^N)$ since Theorem \ref{Ppbcm} indicates  $\mathcal{D}^{1,p}_{\alpha,r}(\mathbb{R}^N)\hookrightarrow L_\beta^{p^*_{\alpha,\beta}}(\mathbb{R}^N)$ continuously.

    Then we split this section into two subsections in order to prove Theorem \ref{thmprtp} and Theorem \ref{thmprtp2}, respectively.

    \subsection{The case $2\leq p<N$.}\label{subsectpb2}
    \

    Firstly, we are going to give the following spectral gap-type estimate.
    \begin{lemma}\label{lemsgap}
    Let $2\leq p<N$. Given any $\gamma_0>0$, there exists $\overline{\delta}=\overline{\delta}(N,p,\alpha,\beta,\gamma_0)>0$ such that for any function $v\in \mathcal{D}^{1,p}_{\alpha,r}(\mathbb{R}^N)$ orthogonal to $T_{U} \mathcal{M}$ in $L^2_{\beta,*}(\mathbb{R}^N)$ satisfying $\|v\|\leq \overline{\delta}$, we have
    \begin{equation*}
    \begin{split}
    & \int_{\mathbb{R}^N}|x|^{\alpha}\left[|\nabla U|^{p-2}|\nabla v|^2+(p-2)|\omega|^{p-2}(|\nabla (U+v)|-|\nabla U|)^2\right]dx \\
    \geq &\left[(p^*_{\alpha,\beta}-1)+\tau\right]
    \int_{\mathbb{R}^N}|x|^{\beta}U^{p^*_{\alpha,\beta}-2}|v|^2dx,
    \end{split}
    \end{equation*}
    where $\tau>0$ is given in Proposition \ref{propevl}, and $\omega: \mathbb{R}^{2N}\to \mathbb{R}^N$ is defined in analogy to Lemma \ref{lemui1p}:
    \begin{eqnarray*}
    \omega=\omega(\nabla U,\nabla (U+v))=
    \left\{ \arraycolsep=1.5pt
       \begin{array}{ll}
        \nabla U,\ \ &{\rm if}\ \ |\nabla U|<|\nabla (U+v)|,\\[3mm]
        \left(\frac{|\nabla (U+v)|}{|\nabla U|}\right)^{\frac{1}{p-2}}\nabla (U+v),\ \ &{\rm if}\ \  |\nabla (U+v)|\leq |\nabla U|.
        \end{array}
    \right.
    \end{eqnarray*}
    \end{lemma}

    \begin{proof}
    If the statement of this lemma fails, then there exists a sequence $0\not\equiv v_i\to 0$ in $\mathcal{D}^{1,p}_{\alpha,r}(\mathbb{R}^N)$, with $v_i$ orthogonal to $T_{U} \mathcal{M}$, such that
    \begin{align}\label{evbc}
    & \int_{\mathbb{R}^N}|x|^{\alpha}\left[|\nabla U|^{p-2}|\nabla v_i|^2+(p-2)|\omega_i|^{p-2}(|\nabla (U+v_i)|-|\nabla U|)^2\right]dx \nonumber\\
    < & \left[(p^*_{\alpha,\beta}-1)+\tau\right]
    \int_{\mathbb{R}^N}|x|^{\beta}U^{p^*_{\alpha,\beta}-2}|v_i|^2dx,
    \end{align}
    where $\omega_i$ corresponds to $v_i$ as in the statement.
    Let
    \[
    \varepsilon_i:=\|v_i\|_{\mathcal{D}^{1,2}_{\alpha,*}(\mathbb{R}^N)}
    =\left(\int_{\mathbb{R}^N}|x|^\alpha|\nabla U|^{p-2} |\nabla v_i|^2  dx\right)^{\frac{1}{2}},\quad \widehat{v}_i=\frac{v_i}{\varepsilon_i}.
    \]
    Note that, since $p\geq 2$, it follows by H\"{o}lder inequality that
    \[
    \int_{\mathbb{R}^N}|x|^{\alpha}|\nabla U|^{p-2}|\nabla v_i|^2dx
    \leq \left(\int_{\mathbb{R}^N}|x|^{\alpha}|\nabla U|^{p}dx\right)^{1-\frac{p}{2}}
    \left(\int_{\mathbb{R}^N}|x|^{\alpha}|\nabla v_i|^{p}dx\right)^{\frac{p}{2}}
    \to 0,
    \]
    hence $\varepsilon_i\to 0$, as $i\to \infty$.
    Since $\|\widehat{v}_i\|_{\mathcal{D}^{1,2}_{\alpha,*}(\mathbb{R}^N)}=1$, Proposition \ref{propcet} implies that, up to a subsequence, $\widehat{v}_i\rightharpoonup \widehat{v}$ in $\mathcal{D}^{1,2}_{\alpha,*}(\mathbb{R}^N)$ and $\widehat{v}_i \rightarrow\widehat{v}$ in $L^2_{\beta,*}(\mathbb{R}^N)$ for some $\widehat{v}\in \mathcal{D}^{1,2}_{\alpha,*}(\mathbb{R}^N)$. Also, since $p\geq 2$, it follows from \eqref{evbc} that
    \[
    1=\int_{\mathbb{R}^N}|x|^{\alpha} |\nabla U|^{p-2}|\nabla \widehat{v}_i|^2 dx \leq  \left[(p^*_{\alpha,\beta}-1)+\tau\right]
    \int_{\mathbb{R}^N}|x|^{\beta}U^{p^*_{\alpha,\beta}-2}
    |\widehat{v}_i|^2dx,
    \]
    then we deduce that
    \[
    \|\widehat{v}_i\|_{L^2_{\beta,*}(\mathbb{R}^N)}
    =\int_{\mathbb{R}^N}|x|^{\beta}U^{p^*_{\alpha,\beta}-2}
    |\widehat{v}_i|^2dx\geq c
    \]
    for some $c>0$.

    Fix $R>1$ which can be chosen arbitrarily large, set
    \begin{equation*}
    \begin{split}
    \mathcal{R}_i:=\{2|\nabla U|\geq |\nabla v_i|\}&,\quad \mathcal{S}_i:=\{2|\nabla U|< |\nabla v_i|\},  \\
    \mathcal{R}_{i,R}:=\left(B(0,R)\backslash B(0,1/R)\right)\cap \mathcal{R}_i&,\quad
    \mathcal{S}_{i,R}:=\left(B(0,R)\backslash B(0,1/R)\right)\cap \mathcal{S}_i,
    \end{split}
    \end{equation*}
    thus $B(0,R)\backslash B(0,1/R)=\mathcal{R}_{i,R} \cup \mathcal{S}_{i,R}$.
    Since the integrand in the left hand side of \eqref{evbc} is nonnegative, we have
    \begin{align}\label{evbcb}
    & \int_{B(0,R)\backslash B(0,1/R)}|x|^{\alpha}\left[|\nabla U|^{p-2}|\nabla \widehat{v}_i|^2+(p-2)|\omega_i|^{p-2}\left(\frac{|\nabla (U+v_i)|-|\nabla U|}{\varepsilon_i}\right)^2\right]dx \nonumber\\
    < & \left[(p^*_{\alpha,\beta}-1)+\tau\right]
    \int_{\mathbb{R}^N}|x|^{\beta}U^{p^*_{\alpha,\beta}-2}|\widehat{v}_i|^2dx.
    \end{align}
    From Proposition \ref{propcet} the continuous embedding theorem, we have
    \begin{align*}
    c\leq \int_{\mathbb{R}^N}|x|^{\beta}
    U^{p^*_{\alpha,\beta}-2}|\widehat{v}_i|^2dx
    \leq C_1 \int_{\mathbb{R}^N}|x|^\alpha|\nabla U|^{p-2} |\nabla \widehat{v}_i|^2  dx=C_1,
    \end{align*}
    thus
    \begin{align*}
    & \int_{\mathcal{R}_{i,R}}|x|^\alpha\left| \nabla  U \right|^{p-2}|\nabla \widehat{v}_i|^2 dx
    +\varepsilon^{p-2}_i\int_{\mathcal{S}_{i,R}}|x|^\alpha|\nabla \widehat{v}_i|^p dx \leq C_2.
    \end{align*}
    Then we obtain
    \[
    \varepsilon_i^{-2}\int_{\mathcal{S}_{i,R}}|x|^{\alpha} |\nabla U|^{p}dx
    \leq \frac{\varepsilon_i^{p-2}}{2^p}\int_{\mathcal{S}_{i,R}}|x|^{\alpha} |\nabla \widehat{v}_i|^{p}dx
    \leq C_3,
    \]
    and since
    \[
    0<c(R)\leq |\nabla U|\leq C(R)\quad \mbox{inside}\quad B(0,R)\backslash B(0,1/R),\quad \forall R>1,
    \]
    for some constants $c(R)\leq C(R)$ depending only on $R$, we conclude that
    \begin{align}\label{pb2csirt0}
    |\mathcal{S}_{i,R}|\to 0\quad \mbox{as}\quad i\to \infty,\quad \forall R>1.
    \end{align}
    Now, writing
    \[
    \widehat{v}_i=\widehat{v}+\varphi_i,\quad\mbox{with}\quad \varphi_i:= \widehat{v}_i-\widehat{v},
    \]
    since $R>1$ is arbitrary, we have
    \[
    \varphi_i \rightharpoonup 0 \quad \mbox{locally in}\quad   \mathcal{D}^{1,2}_{\alpha,*}(\mathbb{R}^N\backslash\{0\}).\]
    Moreover, we have $|\omega_i|\to |\nabla U|$ a.e. in $\mathbb{R}^N$. Then, let us rewrite
    \begin{equation*}
    \begin{split}
    \left(\frac{|\nabla (U+v_i)|-|\nabla U|}{\varepsilon_i}\right)^2
    = & \left(\left[\int^1_0\frac{\nabla U+ t\nabla v_i}{|\nabla U+ t\nabla v_i|}dt\right]\cdot \nabla \widehat{v}_i\right)^2 \\
    = & \left(\left[\int^1_0\frac{\nabla U+ t\nabla v_i}{|\nabla U+ t\nabla v_i|}dt\right]\cdot \nabla (\widehat{v}+\varphi_i)\right)^2.
    \end{split}
    \end{equation*}
    Hence, if we set
    \[
    f_{i,1}=\left[\int^1_0\frac{\nabla U+ t\nabla v_i}{|\nabla U+ t\nabla v_i|}dt\right]\cdot \nabla \widehat{v},\quad
    f_{i,2}=\left[\int^1_0\frac{\nabla U+ t\nabla v_i}{|\nabla U+ t\nabla v_i|}dt\right]\cdot \nabla \varphi_i,
    \]
    since $\frac{\nabla U+ t\nabla v_i}{|\nabla U+ t\nabla v_i|}\to \frac{\nabla U}{|\nabla U|}$ a.e., it follows from Lebesgue's dominated convergence theorem that
    \[
    f_{i,1}\to \frac{\nabla U}{|\nabla U|}\cdot\nabla \widehat{v}\quad \mbox{locally in}\quad L^2(\mathbb{R}^N\backslash\{0\}),\quad f_{i,2}\chi_{\mathcal{R}_i}\rightharpoonup 0\quad \mbox{locally in}\quad L^2(\mathbb{R}^N\backslash\{0\}).
    \]
    Thus, the left hand side of \eqref{evbcb} from below as follows:
    \begin{align}\label{evbcbb}
    & \int_{\mathcal{R}_{i,R}}|x|^{\alpha}\left[|\nabla U|^{p-2}|\nabla \widehat{v}_i|^2+(p-2)|\omega_i|^{p-2}\left(\frac{|\nabla (U+v_i)|-|\nabla U|}{\varepsilon_i}\right)^2\right]dx \nonumber\\
    = & \int_{\mathcal{R}_{i,R}}|x|^{\alpha}\left[|\nabla U|^{p-2}\left(|\nabla \widehat{v}|^2+2 \nabla \varphi_i\cdot \nabla \widehat{v}\right)+(p-2)|\omega_i|^{p-2}\left(f_{i,1}^2+2f_{i,1}f_{i,2}\right)\right]dx \nonumber\\
    & + \int_{\mathcal{R}_{i,R}}|x|^{\alpha}\left[|\nabla U|^{p-2} |\nabla \varphi|^2+(p-2)|\omega_i|^{p-2}f_{i,2}^2\right]dx \nonumber\\
    \geq & \int_{\mathcal{R}_{i,R}}|x|^{\alpha}\left[|\nabla U|^{p-2}\left(|\nabla \widehat{v}|^2+2 \nabla \varphi_i\cdot \nabla \widehat{v}\right)+(p-2)|\omega_i|^{p-2}
    \left(f_{i,1}^2+2f_{i,1}f_{i,2}\right)\right]dx.
    \end{align}
    Then, combining the convergence
    \begin{equation*}
    \begin{split}
    & \nabla \varphi_i \chi_{\mathcal{R}_i}\rightharpoonup 0,
    \quad f_{i,1}\to \frac{\nabla U}{|\nabla U|}\cdot\nabla \widehat{v},
    \quad f_{i,2}\chi_{\mathcal{R}_i}\rightharpoonup 0,
    \quad \mbox{locally in}\quad L^2(\mathbb{R}^N\backslash\{0\}),\\
    & |\omega_i|\to |\nabla U|\quad \mbox{a.e.}, \quad |(B(0,R)\backslash B(0,1/R))\backslash\mathcal{R}_{i,R}|=|\mathcal{S}_{i,R}|\to 0,
    \end{split}
    \end{equation*}
    with the fact that
    \[
    |\omega_i|^{p-2}\leq C(p)|\nabla U|^{p-2},
    \]
    by Lebesgue's dominated convergence theorem, we deduce that
    \begin{small}
    \begin{align*}
    & \lim_{i\to \infty}\int_{\mathcal{R}_{i,R}}|x|^{\alpha}\left[|\nabla U|^{p-2}\left(|\nabla \widehat{v}|^2+2 \nabla \varphi_i\cdot \nabla \widehat{v}\right)+(p-2)|\omega_i|^{p-2}
    \left(f_{i,1}^2+2f_{i,1}f_{i,2}\right)\right]dx
    \\
    \to & \int_{B(0,R)\backslash B(0,1/R)}|x|^{\alpha}\left[|\nabla U|^{p-2}|\nabla \widehat{v} |^2+(p-2)|\nabla U|^{p-2}\left(\frac{\nabla U\cdot\nabla \widehat{v} }{| \nabla U|}\right)^2\right]dx,
    \end{align*}
    \end{small}
    then combining \eqref{evbcb} with \eqref{evbcbb} we have
    \begin{small}
    \begin{align}\label{evbcbbcf}
    & \liminf_{i\to \infty}\int_{B(0,R)\backslash B(0,1/R)}|x|^{\alpha}\left[|\nabla U|^{p-2}|\nabla \widehat{v}_i|^2+(p-2)|\omega_i|^{p-2}\left(\frac{|\nabla (U+v_i)|-|\nabla U|}{\varepsilon_i}\right)^2\right]dx
    \nonumber\\
    \geq & \int_{B(0,R)\backslash B(0,1/R)}|x|^{\alpha}\left[|\nabla U|^{p-2}|\nabla \widehat{v} |^2+(p-2)|\nabla U|^{p-2}\left(\frac{\nabla U\cdot\nabla \widehat{v} }{| \nabla U|}\right)^2\right]dx.
    \end{align}
    \end{small}
    Recalling \eqref{evbcb} and since $R>1$ is arbitrary, \eqref{evbcbbcf} proves that
    \begin{align*}
    & \int_{\mathbb{R}^N}|x|^{\alpha}\left[|\nabla U|^{p-2}|\nabla \widehat{v}|^2+(p-2)|\nabla U|^{p-2}\left(\frac{\nabla U\cdot\nabla \widehat{v}}{|\nabla U|}\right)^2\right]dx \\
    \leq & \left[(p^*_{\alpha,\beta}-1)+\tau\right]
    \int_{\mathbb{R}^N}|x|^{\beta}U^{p^*_{\alpha,\beta}-2}|\widehat{v}|^2dx,
    \end{align*}
    which contradicts Proposition \ref{propevl} due to the orthogonality of $\widehat{v}$ to $T_{U} \mathcal{M}$ (being the strong $L^2_{\beta,*}(\mathbb{R}^N)$ limit of $\widehat{v}_i$).
    \end{proof}

    The main ingredient in the proof of Theorem \ref{thmprtp} is contained in the lemma below, where the behavior near $\mathcal{M}$ is studied.

    \begin{lemma}\label{lemma:rtnm2b}
    Suppose $2\leq p<N$.
    There exists a small constant $\rho>0$ such that for any sequence $\{u_n\}\subset \mathcal{D}^{1,p}_{\alpha,r}(\mathbb{R}^N)\backslash \mathcal{M}$ satisfying $\inf_n\|u_n\|>0$ and ${\rm dist}(u_n,\mathcal{M})\to 0$, it holds that
    \begin{equation}\label{rtnmb}
    \liminf_{n\to\infty}
    \frac{\|u_n\|^p- S_r\|u_n\|_*^p}
    {{\rm dist}(u_n,\mathcal{M})^p}
    \geq \rho.
    \end{equation}
    \end{lemma}

    \begin{proof}
    Let $d_n:={\rm dist}(u_n,\mathcal{M})=\inf_{c\in\mathbb{R}, \lambda>0}\|u_n-cU_\lambda\|\to 0$ as $n\to \infty$. We know that for each $u_n\in \mathcal{D}^{1,p}_{\alpha,r}(\mathbb{R}^N)$, there exist $c_n\in\mathbb{R}$ and $\lambda_n>0$ such that $d_n=\|u_n-c_nU_{\lambda_n}\|$. In fact, since $2\leq p<N$, for each fixed $n$, from Lemma \ref{lemui1p}, we obtain that for any $0<\kappa<1$, there exists a constant $\mathcal{C}_1=\mathcal{C}_1(p,\kappa)>0$ such that
    \begin{small}
    \begin{align}\label{ikeda}
    \|u_n-cU_\lambda\|^p
    = & \int_{\mathbb{R}^N}|x|^{\alpha}|\nabla u_n-c\nabla U_\lambda|^p dx\nonumber\\
    \geq & \int_{\mathbb{R}^N}|x|^{\alpha}|\nabla u_n|^p dx
    -pc\int_{\mathbb{R}^N}|x|^{\alpha}|\nabla u_n|^{p-2}  \nabla u_n\cdot \nabla U_\lambda dx  +\mathcal{C}_1|c|^{p}\int_{\mathbb{R}^N}|x|^{\alpha}|\nabla U_\lambda|^{p} dx
    \nonumber\\ &
    +\frac{(1-\kappa)p}{2}c^2 \int_{\mathbb{R}^N}|x|^{\alpha} |\nabla u_n|^{p-2}  |\nabla U_\lambda|^2dx \nonumber\\
    & +\frac{(1-\kappa)p(p-2)}{2}\int_{\mathbb{R}^N}|x|^{\alpha}|\omega(\nabla u_n, \nabla u_n-c \nabla U_{\lambda})|^{p-2}(|c \nabla U_{\lambda}|-|\nabla u_n|)^2  dx \nonumber\\
    \geq & \|u_n\|^p+ \mathcal{C}_1|c|^{p}\|U\|^p-pc\int_{\mathbb{R}^N}|x|^{\alpha}|\nabla u_n|^{p-2}  \nabla u_n\cdot \nabla U_\lambda dx\nonumber\\
    \geq & \|u_n\|^p+ \mathcal{C}_1|c|^{p}\|U\|^p-p|c|\|U\|\|u_n\|^{p-1},
    \end{align}
    \end{small}
    where $\omega:\mathbb{R}^{2N}\to \mathbb{R}^N$ corresponds to $\nabla u_n$ and $\nabla u_n-c \nabla U_{\lambda}$ as in Lemma \ref{lemui1p}  for the case $p\geq 2$.
    Thus the minimizing sequence of $d_n$, say $\{c_{n,m},\lambda_{n,m}\}$, must satisfying $|c_{n,m}|\leq C$ which means $\{c_{n,m}\}$ is bounded.
    On the other hand,
    \begin{equation*}
    \begin{split}
    \left|\int_{|\lambda x|\leq \rho}|x|^{\alpha}|\nabla u_n|^{p-2}  \nabla u_n\cdot \nabla U_\lambda dx\right|
    \leq & \int_{|y|\leq \rho}|y|^{\alpha}|\nabla  (u_n)_{\frac{1}{\lambda}}(y)|^{p-1}|\nabla U(y)| dy \\
    \leq & \|u_n\|^{p-1}\left(\int_{|y|\leq \rho}|y|^{\alpha}|\nabla U|^p dy\right)^{\frac{1}{p}} \\
    = & o_\rho(1)
    \end{split}
    \end{equation*}
    as $\rho\to 0$ which is uniform for $\lambda>0$, where $(u_n)_{\frac{1}{\lambda}}(y)=\lambda^{-\frac{N-p+\alpha}{p}}u_n(\lambda^{-1}y)$, and
    \begin{equation*}
    \begin{split}
    \left|\int_{|\lambda x|\geq \rho}|x|^{\alpha}|\nabla u_n|^{p-2}  \nabla u_n\cdot \nabla U_\lambda dx\right|
    \leq & \|U\|\left(\int_{|x|\geq \frac{\rho}{\lambda}}|x|^{\alpha}|\nabla u_n|^p dy\right)^{\frac{1}{p}}
    =  o_\lambda(1)
    \end{split}
    \end{equation*}
    as $\lambda\to 0$ for any fixed $\rho>0$. By taking $\lambda\to 0$ and then $\rho\to 0$, we obtain
    \[\left|\int_{\mathbb{R}^N}|x|^{\alpha}|\nabla u_n|^{p-2}  \nabla u_n\cdot \nabla U_\lambda dx\right| \to 0\quad \mbox{as}\quad \lambda\to 0.\]
    Moreover, by the explicit from of $U_\lambda$ we have
    \begin{equation*}
    \begin{split}
    \left|\int_{|\lambda x|\leq R}|x|^{\alpha}|\nabla u_n|^{p-2}  \nabla u_n\cdot \nabla U_\lambda dx\right|
    \leq & \|U\|\left(\int_{| x|\leq \frac{R}{\lambda}}|x|^{\alpha}|\nabla u_n|^p dx\right)^{\frac{1}{p}}
    = o_\lambda(1)
    \end{split}
    \end{equation*}
    as $\lambda\to +\infty$ for any fixed $R>0$, and
    \begin{equation*}
    \begin{split}
    \left|\int_{|\lambda x|\geq R}|x|^{\alpha}|\nabla u_n|^{p-2}  \nabla u_n\cdot \nabla U_\lambda dx\right|
    \leq & \int_{|y|\geq R}|y|^{\alpha}|\nabla (u_n)_{\frac{1}{\lambda}}(y)|^{p-1}|\nabla U(y)| dy \\
    \leq & \|u_n\|^{p-1}\left(\int_{|y|\geq R}|y|^{\alpha}|\nabla U|^p dy\right)^{\frac{1}{p}}
    =  o_R(1)
    \end{split}
    \end{equation*}
    as $R\to +\infty$ which is uniform for $\lambda>0$. Thus, by taking first $\lambda\to +\infty$ and then $R\to +\infty$, we also obtain
    \[\left|\int_{\mathbb{R}^N}|x|^{\alpha}|\nabla u_n|^{p-2}  \nabla u_n\cdot \nabla U_\lambda dx\right| \to 0\quad \mbox{as}\quad \lambda\to +\infty.\]
    It follows from (\ref{ikeda}) and $d_n\to 0$, $\inf_n\|u_n\|>0$ that the minimizing sequence $\{c_{n,m},\lambda_{n,m}\}$ must satisfying $1/C\leq |\lambda_{n,m}|\leq C$ for some $C>1$,  which means $\{\lambda_{n,m}\}$ is bounded. Thus for each $u_n\in  \mathcal{D}^{1,p}_{\alpha,r}(\mathbb{R}^N)\backslash \mathcal{M}$, $d_n$ can be attained by some $c_n\in\mathbb{R}$ and $\lambda_n>0$.

    Since $\mathcal{M}$ is two-dimensional manifold embedded in $\mathcal{D}^{1,p}_{\alpha,r}(\mathbb{R}^N)$, that is
    \[
    (c,\lambda)\in\mathbb{R}\times\mathbb{R}_+\to cU_\lambda\in \mathcal{D}^{1,p}_{\alpha,r}(\mathbb{R}^N),
    \]
    then from Proposition \ref{propev}, the tangential space at $(c_n,\lambda_n)$ is given by
    \[
    T_{c_n U_{\lambda_n}}\mathcal{M}={\rm Span}\left\{U_{\lambda_n}, \quad \lambda_n^{\frac{N-p+\alpha}{p}}W_0(\lambda_n x),\quad \lambda_n^{\frac{N-p+\alpha}{p}}W_{k,i}(\lambda_n x), i=1,\ldots, M_k\right\},
    \]
    if (\ref{npkev}) holds, otherwise
    \[
    T_{c_n U_{\lambda_n}}\mathcal{M}={\rm Span}\left\{U_{\lambda_n}, \quad \lambda_n^{\frac{N-p+\alpha}{p}}W_0(\lambda_n x)\right\}.
    \]
    Anyway we must have that $(u_n-c_n U_{\lambda_n})$ is perpendicular to $T_{c_n U_{\lambda_n}}\mathcal{M}$, particularly
    \[
    \int_{\mathbb{R}^N}|x|^{\beta}U_{\lambda_n}^{p^*_{\alpha,\beta}-2}(u_n-c_n U_{\lambda_n})\xi dx=0,\quad \forall \xi \in T_{c_n U_{\lambda_n}}\mathcal{M},
    \]
    thus taking $\xi=U_{\lambda_n}$ we obatin
    \begin{equation*}
    \int_{\mathbb{R}^N}|x|^\alpha |\nabla U_{\lambda_n}|^{p-2}\nabla U_{\lambda_n}\cdot \nabla (u_n-c_n U_{\lambda_n}) dx=0.
    \end{equation*}
    Let
    \begin{equation}\label{defunwn}
    u_n=c_n U_{\lambda_n}+d_n w_n,
    \end{equation}
     then $w_n$ is perpendicular to $T_{c_n U_{\lambda_n}}\mathcal{M}$, we have
    \begin{equation*}
    \|w_n\|=1\quad \mbox{and}\quad \int_{\mathbb{R}^N}|x|^\alpha |\nabla U_{\lambda_n}|^{p-2}\nabla U_{\lambda_n}\cdot \nabla w_n dx=0.
    \end{equation*}
    From Lemma \ref{lemui1p}, for any $\kappa>0$, there exists a constant $\mathcal{C}_1=\mathcal{C}_1(p,\kappa)>0$ such that
    \begin{small}
    \begin{align}\label{epknug}
    \int_{\mathbb{R}^N}|x|^{\alpha}|\nabla u_n|^p dx
    = & \int_{\mathbb{R}^N}|x|^{\alpha}|c_n \nabla U_{\lambda_n}+d_n \nabla w_n|^p dx\nonumber\\
    \geq & |c_n|^{p}\int_{\mathbb{R}^N}|x|^{\alpha}|\nabla U_{\lambda_n}|^p dx
    +p|c_n|^{p-2}c_nd_n \int_{\mathbb{R}^N}|x|^{\alpha}|\nabla U_{\lambda_n}|^{p-2}  \nabla U_{\lambda_n}\cdot \nabla w_n dx  \nonumber\\
    & +\mathcal{C}_1d_n^{p}\int_{\mathbb{R}^N}|x|^{\alpha}|\nabla w_n|^{p} dx
    +\frac{(1-\kappa)p}{2} |c_n|^{p-2}d_n^2\int_{\mathbb{R}^N}|x|^{\alpha} |\nabla U_{\lambda_n}|^{p-2}  |\nabla w_n|^2dx \nonumber\\
    & +\frac{(1-\kappa)p(p-2)}{2}\int_{\mathbb{R}^N}|x|^{\alpha}|\omega(c_n \nabla U_{\lambda_n},\nabla u_n)|^{p-2}(|c_n \nabla U_{\lambda_n}|-|\nabla u_n|)^2  dx \nonumber\\
    = & |c_n|^{p}\|U\|^p+ \mathcal{C}_1d_n^{p}
     +\frac{(1-\kappa)p}{2} |c_n|^{p-2}d_n^2\int_{\mathbb{R}^N}|x|^{\alpha} |\nabla U_{\lambda_n}|^{p-2}  |\nabla w_n|^2dx \nonumber\\
    & +\frac{(1-\kappa)p(p-2)}{2}\int_{\mathbb{R}^N}|x|^{\alpha}|\omega(c_n \nabla U_{\lambda_n},\nabla u_n)|^{p-2}(|c_n \nabla U_{\lambda_n}|-|\nabla u_n|)^2  dx,
    \end{align}
    \end{small}
    where $\omega:\mathbb{R}^{2N}\to \mathbb{R}^N$ corresponds to $c_n \nabla U_{\lambda_n}$ and $u_n$ as in Lemma \ref{lemui1p}, since
    \begin{equation*}
    \int_{\mathbb{R}^N}|x|^{\beta}U_{\lambda_n}^{p^*_{\alpha,\beta}-1}w_n dx=\int_{\mathbb{R}^N}|x|^\alpha |\nabla U_{\lambda_n}|^{p-2}\nabla U_{\lambda_n}\cdot \nabla w_n dx=0,
    \end{equation*}
    and
    \begin{equation*}
    \int_{\mathbb{R}^N}|x|^{\beta}U_{\lambda_n}^{p^*_{\alpha,\beta}} dx=\int_{\mathbb{R}^N}|x|^\alpha |\nabla U_{\lambda_n}|^{p}dx=\|U\|^p.
    \end{equation*}
    Then from Lemma \ref{lemui1p*l}, for any $\kappa>0$, there exists a constant $\mathcal{C}_2=\mathcal{C}_2(p^*_{\alpha,\beta},\kappa)>0$ such that
    \begin{align*}
    & \int_{\mathbb{R}^N}|x|^{\beta}|u_n|^{p^*_{\alpha,\beta}} dx \\
    = & \int_{\mathbb{R}^N}|x|^{\beta}|c_n U_{\lambda_n}+d_nw_n|^{p^*_{\alpha,\beta}}  dx\\
    \leq & |c_n|^{p^*_{\alpha,\beta}}\int_{\mathbb{R}^N}|x|^{\beta}U_{\lambda_n}^{p^*_{\alpha,\beta}} dx
    +|c_n|^{p^*_{\alpha,\beta}-2}c_n p^*_{\alpha} d_n \int_{\mathbb{R}^N}|x|^{\beta}U_{\lambda_n}^{p^*_{\alpha,\beta}-1}w_n dx  \\
    & +\left(\frac{p^*_{\alpha,\beta}(p^*_{\alpha,\beta}-1)}{2}+\kappa\right)|c_n|^{p^*_{\alpha,\beta}-2} d_n^2
    \int_{\mathbb{R}^N}|x|^{\beta}U_{\lambda_n}^{p^*_{\alpha,\beta}-2}w_n^2 dx \\
    & +\mathcal{C}_2d_n^{p^*_{\alpha,\beta}}\int_{\mathbb{R}^N}|x|^{\beta}|w_n|^{p^*_{\alpha,\beta}} dx \\
    = & |c_n|^{p^*_{\alpha,\beta}}\|U\|^p
    +\left(\frac{p^*_{\alpha,\beta}(p^*_{\alpha,\beta}-1)}{2}+\kappa\right)|c_n|^{p^*_{\alpha,\beta}-2} d_n^2
    \int_{\mathbb{R}^N}|x|^{\beta}U_{\lambda_n}^{p^*_{\alpha,\beta}-2}w_n^2 dx
    + o(d_n^p),
    \end{align*}
    since $p<p^*_{\alpha,\beta}$.
    Thus, by the concavity of $t\mapsto t^{\frac{p}{p^*_{\alpha,\beta}}}$, we have
    \begin{align}\label{epkeyiyxbb}
    \left(\int_{\mathbb{R}^N}|x|^{\beta}|u_n|^{p^*_{\alpha,\beta}} dx\right)^{\frac{p}{p^*_{\alpha,\beta}}}
    \leq &  |c_n|^p\|U\|^{\frac{p^2}{p^*_{\alpha,\beta}}} +o(d_n^p)
    + \frac{p|c_n|^{p^*_{\alpha,\beta}-2} d_n^2}{p^*_{\alpha,\beta}}\left(\frac{p^*_{\alpha,\beta}(p^*_{\alpha,\beta}-1)}{2}+\kappa\right)
    \nonumber\\
    &\quad \times\|U\|^{\frac{p^2}{p^*_{\alpha,\beta}}-p}\int_{\mathbb{R}^N}|x|^{\beta}U_{\lambda_n}^{p^*_{\alpha,\beta}-2}w_n^2 dx.
    \end{align}
    Therefore, as $d_n\to 0$, combining \eqref{epknug} with \eqref{epkeyiyxbb} , it follows from Lemma \ref{lemsgap} that, by choosing $\kappa>0$ small enough,
    \begin{small}
    \begin{align*}
    & \int_{\mathbb{R}^N}|x|^{\alpha}|\nabla u_n|^p dx- S_r\left(\int_{\mathbb{R}^N}|x|^{\beta}|u_n|^{p^*_{\alpha,\beta}} dx\right)^{\frac{p}{p^*_{\alpha,\beta}}} \\
    \geq & |c_n|^{p}\|U\|^p+ \mathcal{C}_1d_n^{p}
    +\frac{(1-\kappa)p}{2} d_n^2\int_{\mathbb{R}^N}|x|^{\alpha} |\nabla c_n U_{\lambda_n}|^{p-2}  |\nabla w_n|^2dx \\
    & +\frac{(1-\kappa)p(p-2)}{2}\int_{\mathbb{R}^N}|x|^{\alpha}|\omega(c_n \nabla U_{\lambda_n},\nabla u_n)|^{p-2}(|c_n \nabla U_{\lambda_n}|-|\nabla u_n|)^2  dx \\
    & -S_r\Bigg\{|c_n|^p\|U\|^{\frac{p^2}{p^*_{\alpha,\beta}}} +o(d_n^p) \\
    & + \frac{p|c_n|^{p^*_{\alpha,\beta}-2} d_n^2}{p^*_{\alpha,\beta}}\left(\frac{p^*_{\alpha,\beta}(p^*_{\alpha,\beta}-1)}{2}+\kappa\right)
    \|U\|^{\frac{p^2}{p^*_{\alpha,\beta}}-p}\int_{\mathbb{R}^N}|x|^{\beta}U_{\lambda_n}^{p^*_{\alpha,\beta}-2}w_n^2 dx\Bigg\} \\
    \geq  & \mathcal{C}_1d_n^p -o(d_n^p)  \\
    & + \frac{(1-\kappa)p  |c_n|^{p^*_{\alpha,\beta}-2} d_n^2}{2}\left[(p^*_{\alpha,\beta}-1)+\tau\right]\int_{\mathbb{R}^N}|x|^{\beta}U_{\lambda_n}^{p^*_{\alpha,\beta}-2}w_n^2 dx\\
    & - \frac{p|c_n|^{p^*_{\alpha,\beta}-2} d_n^2}{p^*_{\alpha,\beta}}\left(\frac{p^*_{\alpha,\beta}(p^*_{\alpha,\beta}-1)}{2}+\kappa\right)
    S_r\|U\|^{\frac{p^2}{p^*_{\alpha,\beta}}-p}\int_{\mathbb{R}^N}|x|^{\beta}U_{\lambda_n}^{p^*_{\alpha,\beta}-2}w_n^2 dx \\
    \geq  & \mathcal{C}_1d_n^p -o(d_n^p),
    \end{align*}
    \end{small}
    since $\|U\|^p=\|U\|_*^{p^*_{\alpha,\beta}}=S_r^{\frac{p^*_{\alpha,\beta}}{p^*_{\alpha,\beta}-p}}$ and $\|U\|^p=S_r\|U\|_*^p$, then (\ref{rtnmb}) follows immediately.
    \end{proof}

\noindent{\bf Proof of Theorem \ref{thmprtp}.} We argue by contradiction. In fact, if the theorem is false then there exists a sequence $\{u_n\}\subset \mathcal{D}^{1,p}_{\alpha,r}(\mathbb{R}^N)\backslash \mathcal{M}$ such that
    \begin{equation*}
    \frac{\|u_n\|^p-S_r\|u_n\|^p_*}{{\rm dist}(u_n,\mathcal{M})^p}\to 0,\quad \mbox{as}\quad n\to \infty.
    \end{equation*}
    By homogeneity, we can assume that $\|u_n\|=1$, and after selecting a subsequence we can assume that ${\rm dist}(u_n,\mathcal{M})\to \xi\in[0,1]$ since ${\rm dist}(u_n,\mathcal{M})=\inf_{c\in\mathbb{R}, \lambda>0}\|u_n-cU_{\lambda}\|\leq \|u_n\|$. If $\xi=0$, then we have a contradiction by Lemma \ref{lemma:rtnm2b}.

    The other possibility only is that $\xi>0$, that is
    \[{\rm dist}(u_n,\mathcal{M})\to \xi>0\quad \mbox{as}\quad n\to \infty,\]
    then we must have
    \begin{equation}\label{wbsi}
    \|u_n\|^p-S_r\|u_n\|^p_*\to 0,\quad \|u_n\|=1.
    \end{equation}
    Since $\{u_n\}\subset \mathcal{D}^{1,p}_{\alpha,r}(\mathbb{R}^N)\backslash \mathcal{M}$ are radial, making the changes that $v_n(s)=u_n(r)$ with $r=s^{\frac{p}{p+\beta-\alpha}}$, then (\ref{wbsi}) is equivalent to
    \begin{equation}\label{bsiy}
    \begin{split}
    \int^\infty_0|v_n'(s)|^p s^{K-1}ds
    -C_p(K)\left(\int^\infty_0|v_n(s)|^{\frac{Kp}{K-p}}s^{K-1}ds\right)^{\frac{K-p}{K}}\to 0
    \end{split}
    \end{equation}
    where $K=\frac{p(N+\beta)}{p+\beta-\alpha}>p$ and $C_p(K)=\pi^{\frac{p}{2}}K\left(\frac{K-p}{p-1}\right)^{p-1}
    \left(\frac{\Gamma(\frac{K}{p})\Gamma(1+K-K/p)}{\Gamma(K/2+1)\Gamma(K)}\right)^{\frac{p}{K}}
    \left(\frac{\Gamma(K/2)}{2\pi^{K/2}}\right)^{\frac{p}{K}}$, see the proof of Theorem \ref{thmPbcm}. When $K$ is an integer, then (\ref{bsiy}) is equivalent to
    \begin{equation}\label{bsib}
    \begin{split}
    \int_{\mathbb{R}^K}|\nabla v_n|^p dx
    -S(K)\left(\int_{\mathbb{R}^K}|v_n|^{\frac{Kp}{K-p}}dx\right)^{\frac{K-p}{K}}\to 0,
    \end{split}
    \end{equation}
    where $S(K)=\pi^{\frac{p}{2}}K\left(\frac{K-p}{p-1}\right)^{p-1}
    \left(\frac{\Gamma(\frac{K}{p})\Gamma(1+K-K/p)}{\Gamma(K/2+1)\Gamma(K)}\right)^{\frac{p}{K}}$ is the best constant for the embedding of the space $\mathcal{D}^{1,p}_0(\mathbb{R}^K)$ into $L^{\frac{Kp}{K-p}}(\mathbb{R}^K)$, see \cite{Ta76}. In this case, by Lions' concentration and compactness principle (see \cite[Theorem \uppercase\expandafter{\romannumeral 1}.1]{Li85-1}), we have that there exists a sequence of positive numbers $\lambda_n$ such that
    \begin{equation*}
    \lambda_n^{\frac{K-p}{p}}v_n(\lambda_n x)\to V\quad \mbox{in}\quad \mathcal{D}^{1,p}_0(\mathbb{R}^K)\quad \mbox{as}\quad n\to \infty,
    \end{equation*}
    where $V(x)=c(a+|x|^{\frac{p}{p-1}})^{-\frac{K-p}{p}}$ for some $c\neq 0$ and $a>0$, that is
    \begin{equation*}
    \tau_n^{\frac{N-p+\alpha}{p}}u_n(\tau_n x)\to U_*\quad \mbox{in}\quad \mathcal{D}^{1,p}_\alpha(\mathbb{R}^N)\quad \mbox{as}\quad n\to \infty,
    \end{equation*}
    for some $U_*\in\mathcal{M}$, where $\tau_n=\lambda_n^{\frac{p+\beta-\alpha}{p}}$, which implies
    \begin{equation*}
    {\rm dist}(u_n,\mathcal{M})={\rm dist}\left(\tau_n^{\frac{N-p+\alpha}{p}}u_n(\tau_n x),\mathcal{M}\right)\to 0 \quad \mbox{as}\quad n\to \infty,
    \end{equation*}
    this is a contradiction. Moreover, since $\{u_n\}$ are radial, even $K$ is not an integer we can also get analogous contradiction.
    \qed

\subsection{The case $1<p<2$.} \label{sebsectp12}
\

    As mentioned in \cite{FZ22}, we shall see that Proposition \ref{propcet} allows us to deal with the case $\frac{2(N+\alpha)}{N+2+\beta}< p<2$. However, when $1<p\leq\frac{2(N+\alpha)}{N+2+\beta}$ which implies $p<p^*_{\alpha,\beta}\leq 2$, we will need a much more delicate compactness result that we now present.

    \begin{lemma}\label{propcetl}
    Let $1<p\leq \frac{2(N+\alpha)}{N+2+\beta}$ and $p-N<\alpha<p+\beta$, and let $v_i$ be a sequence of radial functions in $\mathcal{D}^{1,p}_{\alpha,r}(\mathbb{R}^N)$ satisfying
    \begin{align}\label{propcetll}
    & \int_{\mathbb{R}^N}|x|^{\alpha}
    (|\nabla U|+\varepsilon_i|\nabla v_i|)^{p-2}|\nabla v_i|^2dx\leq 1,
    \end{align}
    where $\varepsilon_i\in (0,1)$ is a sequence of positive numbers converging to $0$. Then, up to a subsequence, $v_i$ convergence weakly in $\mathcal{D}^{1,p}_{\alpha}(\mathbb{R}^N)$ to some $v\in\mathcal{D}^{1,p}_{\alpha,r}(\mathbb{R}^N)\cap L^2_{\beta,*}(\mathbb{R}^N)$. Also, given any constant $C_1\geq 0$ it holds
    \begin{align}\label{propcetlc}
    \int_{\mathbb{R}^N}|x|^{\beta}
    \frac{(U+C_1|\varepsilon_iv_i|)^{p^*_{\alpha,\beta}}}
    {U^2+|\varepsilon_iv_i|^2}|v_i|^2dx
    \to \int_{\mathbb{R}^N}|x|^{\beta}
    U^{p^*_{\alpha,\beta}-2}|v|^2dx.
    \end{align}
    \end{lemma}

    \begin{proof}
    Observe that, since $1<p<2$, by H\"{o}lder inequality we have
    \begin{small}
    \begin{align*}
    \int_{\mathbb{R}^N}|x|^{\alpha}|\nabla v_i|^pdx
    \leq & \left(\int_{\mathbb{R}^N}|x|^{\alpha}
    (|\nabla U|+\varepsilon_i|\nabla v_i|)^{p-2}|\nabla v_i|^2dx \right)^{\frac{p}{2}}
    \left(\int_{\mathbb{R}^N}|x|^{\alpha}
    (|\nabla U|+\varepsilon_i|\nabla v_i|)^{p}dx \right)^{1-\frac{p}{2}} \\
    \geq & C(N,p,\alpha,\beta)\left(\int_{\mathbb{R}^N}|x|^{\alpha}
    (|\nabla U|+\varepsilon_i|\nabla v_i|)^{p-2}|\nabla v_i|^2dx \right)^{\frac{p}{2}}
    \\ &\times
    \left(1+\varepsilon_i^p\int_{\mathbb{R}^N}|x|^{\alpha}|\nabla v_i|^{p}dx \right)^{1-\frac{p}{2}},
    \end{align*}
    \end{small}
    that combined with \eqref{propcetll} gives
    \begin{small}
    \begin{align*}
    \left(\int_{\mathbb{R}^N}|x|^{\alpha}|\nabla v_i|^pdx\right)^{\frac{2}{p}}
    \leq & C(N,p,\alpha,\beta)\int_{\mathbb{R}^N}|x|^{\alpha}
    (|\nabla U|+\varepsilon_i|\nabla v_i|)^{p-2}|\nabla v_i|^2dx
    \leq C(N,p,\alpha,\beta).
    \end{align*}
    \end{small}
    Thus, up to a subsequence, $v_i$ convergence weakly in $\mathcal{D}^{1,p}_{\alpha}(\mathbb{R}^N)$ and also a.e. to some $v\in\mathcal{D}^{1,p}_{\alpha,r}(\mathbb{R}^N)$. Hence, to conclude the proof, we need to show the validity of \eqref{propcetlc}.

    Since $\{v_i\}$ are radial, let $v_i(r)=w_i(s)$, where $r=s^{t}$ with $t=\frac{p}{p+\beta-\alpha}$, then
    \begin{align*}
    \int_{\mathbb{R}^N}|x|^{\beta}
    \frac{(U+C_1|\varepsilon_iv_i|)^{p^*_{\alpha,\beta}}}
    {U^2+|\varepsilon_iv_i|^2}|v_i|^2dx
    = & \omega_{N-1}t\int^\infty_0 \frac{(V(s)+C_1|\varepsilon_iw_i(s)|)^{\frac{Kp}{K-p}}}
    {V^2(s)+|\varepsilon_iw_i(s)|^2}|w_i(s)|^2s^{K-1}ds,
    \end{align*}
    and
    \begin{align*}
    \int_{\mathbb{R}^N}|x|^{\beta}
    U^{p^*_{\alpha,\beta}-2}|v|^2dx
    = & \omega_{N-1}t\int^\infty_0 V^{\frac{Kp}{K-p}-2}|w(s)|^2s^{K-1}ds,
    \end{align*}
    where $K=\frac{p(N+\beta)}{p+\beta-\alpha}>p$ and $V(s)=U(r)$.
    From \cite[Lemma 3.4]{FZ22} for the radial case, it holds that
    \begin{align*}
    \int^\infty_0 \frac{(V(s)+C_1|\varepsilon_iw_i(s)|)^{\frac{Kp}{K-p}}}
    {V^2(s)+|\varepsilon_iw_i(s)|^2}|w_i(s)|^2s^{K-1}ds
     \to & \int^\infty_0 V^{\frac{Kp}{K-p}-2}|w(s)|^2s^{K-1}ds,
    \end{align*}
    thus \eqref{propcetlc} holds.
    \end{proof}

    An important consequence of Lemma \ref{propcetl} is the following weighted Orlicz-type Poincar\'{e} inequality:
    \begin{corollary}\label{propcetlpi}
    Let $1<p\leq \frac{2(N+\alpha)}{N+2+\beta}$ and $p-N<\alpha<p+\beta$. There exists $\varepsilon_0=\varepsilon_0(N,p,\alpha,\beta)>0$ small such that the following holds:
    For any $\varepsilon \in (0,\varepsilon_0)$ and any radial function $v\in \mathcal{D}^{1,p}_{\alpha,r}(\mathbb{R}^N)\cap \mathcal{D}^{1,2}_{\alpha,*}(\mathbb{R}^N)$ satisfying
    \begin{align*}
    & \int_{\mathbb{R}^N}|x|^{\alpha}
    (|\nabla U|+\varepsilon|\nabla v|)^{p-2}|\nabla v|^2dx\leq 1,
    \end{align*}
    we have
    \begin{align}\label{propcetlcpi}
    \int_{\mathbb{R}^N}|x|^{\beta}
    (U+|\varepsilon v |)^{p^*_{\alpha,\beta}-2}|v |^2dx
    \leq C(N,p,\alpha,\beta)
    \int_{\mathbb{R}^N}|x|^{\alpha}
    (|\nabla U|+\varepsilon|\nabla v|)^{p-2}|\nabla v|^2dx.
    \end{align}
    \end{corollary}

    \begin{proof}
    Since $v\in \mathcal{D}^{1,p}_{\alpha,r}(\mathbb{R}^N)$ is radial, let $v(r)=w(s)$, where $|x|=r=s^{t}$ with $t=\frac{p}{p+\beta-\alpha}$, then
    \begin{align*}
    \int_{\mathbb{R}^N}|x|^{\beta}
    (U+|\varepsilon v |)^{p^*_{\alpha,\beta}-2}|v |^2dx
    = & \omega_{N-1}t\int^\infty_0
    (V(s)+|\varepsilon w(s)|)^{\frac{Kp}{K-p}-2}|w(s)|^2s^{K-1}ds,
    \end{align*}
    and
    \begin{align*}
    \int_{\mathbb{R}^N}|x|^{\alpha}
    (|\nabla U|+\varepsilon|\nabla v|)^{p-2}|\nabla v|^2dx
    = & \omega_{N-1}t^{1-p}\int^\infty_0 (|V'(s)|+\varepsilon|w'(s)|)^{p-2}|w'(s)|^2s^{K-1}ds,
    \end{align*}
    where where $\omega_{N-1}$ is the surface area for unit ball of $\mathbb{R}^N$, $K=\frac{p(N+\beta)}{p+\beta-\alpha}>p$ and $V(s)=U(r)$.
    From \cite[Corollary 3.5]{FZ22} for the radial case,  
    we deduce \eqref{propcetlcpi} directly.  
    \end{proof}

    We are going to give the following spectral gap-type estimate.
    \begin{lemma}\label{lemsgap2}
    Let $1<p<2\leq N$ and $p-N<\alpha<p+\beta$. Given any $\gamma_0>0$, $C_1>0$ there exists $\overline{\delta}=\overline{\delta}(N,p,\alpha,\beta,\gamma_0,C_1)>0$ such that for any function $v\in \mathcal{D}^{1,p}_{\alpha,r}(\mathbb{R}^N)$ orthogonal to $T_{U} \mathcal{M}$ in $L^2_{\beta,*}(\mathbb{R}^N)$ satisfying $\|v\|\leq \overline{\delta}$, the following holds:
    \begin{itemize}
    \item[$(i)$]
    when $1<p\leq\frac{2(N+\alpha)}{N+2+\beta}$, we have
    \begin{small}
    \begin{align*}
    & \int_{\mathbb{R}^N}|x|^{\alpha}\left[
    |\nabla U|^{p-2}|\nabla v|^2
    +(p-2)|\omega|^{p-2}(|\nabla (U+v)|-|\nabla U|)^2
    +\gamma_0 \min\{|\nabla v|^p,|\nabla U|^{p-2}|\nabla v|^2\}
    \right]dx \\
    \geq & \left[(p^*_{\alpha,\beta}-1)+\tau\right]
    \int_{\mathbb{R}^N}|x|^{\beta}\frac{(U+C_1|v|)^{p^*_{\alpha,\beta}}}{U^2+|v|^2}|v|^2dx;
    \end{align*}
    \end{small}
    \item[$(ii)$]
    when $\frac{2(N+\alpha)}{N+2+\beta}< p<2$, we have
    \begin{small}
    \begin{align*}
    & \int_{\mathbb{R}^N}|x|^{\alpha}\left[
    |\nabla U|^{p-2}|\nabla v|^2
    +(p-2)|\omega|^{p-2}(|\nabla (U+v)|-|\nabla U|)^2
    +\gamma_0 \min\{|\nabla v|^p,|\nabla U|^{p-2}|\nabla v|^2\}
    \right]dx \\
    \geq & \left[(p^*_{\alpha,\beta}-1)+\tau\right]
    \int_{\mathbb{R}^N}|x|^{\beta}U^{p^*_{\alpha,\beta}-2}|v|^2dx,
    \end{align*}
    \end{small}
    \end{itemize}
    where $\tau>0$ is given in Proposition \ref{propevl}, and $\omega: \mathbb{R}^{2N}\to \mathbb{R}^N$ is defined in analogy to Lemma \ref{lemui1p}:
    \begin{eqnarray*}
    \omega=\omega(\nabla U,\nabla (U+v))=
    \left\{ \arraycolsep=1.5pt
       \begin{array}{ll}
        \left(\frac{|\nabla (U+v)|}{(2-p)|\nabla (U+v)|+(p-1)|\nabla U|}\right)^{\frac{1}{p-2}}\nabla U,\ \ &{\rm if}\ \  |\nabla U|<|\nabla (U+v)|\\[3mm]
        \nabla U,\ \ &{\rm if}\ \ |\nabla (U+v)|\leq |\nabla U|
        \end{array}.
    \right.
    \end{eqnarray*}
    \end{lemma}

    \begin{proof}
    We argue by contradiction in these two cases.

    $\bullet$ {\em The case $1<p\leq\frac{2(N+\alpha)}{N+2+\beta}$} which implies $p^*_{\alpha,\beta}\leq2$.
    Suppose the inequality does not hold, then there exists a sequence $0\not\equiv v_i\to 0$ in $\mathcal{D}^{1,p}_{\alpha,r}(\mathbb{R}^N)$, with $v_i$ orthogonal to $T_{U} \mathcal{M}$, such that
    \begin{small}
    \begin{align}\label{evbc2}
    & \int_{\mathbb{R}^N}|x|^{\alpha}\Big[
    |\nabla U|^{p-2}|\nabla v_i|^2
    +(p-2)|\omega_i|^{p-2}(|\nabla (U+v_i)|-|\nabla U|)^2
    \nonumber\\& \quad \quad +\gamma_0 \min\{|\nabla v_i|^p+|\nabla U|^{p-2}|\nabla v_i|^2\}
    \Big]dx \nonumber\\
     <  & \left[(p^*_{\alpha,\beta}-1)+\tau\right]
    \int_{\mathbb{R}^N}|x|^{\beta}\frac{(U+C_1|v_i|)^{p^*_{\alpha,\beta}}}{U^2+|v_i|^2}|v_i|^2dx,
    \end{align}
    \end{small}
    where $\omega_i$ corresponds to $v_i$ as in the statement.
    Let
    \begin{align}\label{defevh}
    \varepsilon_i:=\left(\int_{\mathbb{R}^N}|x|^\alpha(|\nabla U|+|\nabla v_i|)^{p-2}|\nabla v_i|^2 dx\right)^{\frac{1}{2}},\quad \widehat{v}_i=\frac{v_i}{\varepsilon_i}.
    \end{align}
    Note that, since $1<p<2$, it follows by H\"{o}lder inequality that
    \begin{align*}
    \int_{\mathbb{R}^N}|x|^\alpha(|\nabla U|+|\nabla v_i|)^{p-2}|\nabla v_i|^2 dx\leq & \int_{\mathbb{R}^N}|x|^{\alpha}|\nabla U|^{p-2}|\nabla v_i|^2dx \\
    \leq & \left(\int_{\mathbb{R}^N}|x|^{\alpha}|\nabla U|^{p}dx\right)^{1-\frac{p}{2}}
    \left(\int_{\mathbb{R}^N}|x|^{\alpha}|\nabla v_i|^{p}dx\right)^{\frac{p}{2}}
    \\ \to &0,
    \end{align*}
    hence $\varepsilon_i\to 0$, as $i\to \infty$.

    Fix $R>1$ which can be chosen arbitrarily large, set
    \begin{align}\label{defrsi}
    \mathcal{R}_i:=\{2|\nabla U|\geq |\nabla v_i|\}&,\quad \mathcal{S}_i:=\{2|\nabla U|< |\nabla v_i|\},  \nonumber\\
    \mathcal{R}_{i,R}:=\left(B(0,R)\backslash B(0,1/R)\right)\cap \mathcal{R}_i&,\quad
    \mathcal{S}_{i,R}:=\left(B(0,R)\backslash B(0,1/R)\right)\cap \mathcal{S}_i,
    \end{align}
    thus $B(0,R)\backslash B(0,1/R)=\mathcal{R}_{i,R} \cup \mathcal{S}_{i,R}$.
    Since the integrand in the left hand side of \eqref{evbc2} is nonnegative, we have
    \begin{align}\label{evbcb2}
    & \int_{B(0,R)\backslash B(0,1/R)}|x|^{\alpha}\bigg[
    |\nabla U|^{p-2}|\nabla \widehat{v}_i|^2
    +(p-2)|\omega_i|^{p-2}\left(\frac{|\nabla (U+v_i)|-|\nabla U|}{\varepsilon_i}\right)^2 \nonumber\\
    & \quad\quad + \gamma_0 \min\{\varepsilon_i^{p-2}|\nabla \widehat{v}_i|^p,|\nabla U|^{p-2}|\nabla \widehat{v}_i|^2\}
    \bigg]dx \nonumber\\
    < & \left[(p^*_{\alpha,\beta}-1)+\tau\right]
    \int_{\mathbb{R}^N}|x|^{\beta}\frac{(U+C_1|v_i|)^{p^*_{\alpha,\beta}}}{U^2+|v_i|^2}|\widehat{v}_i|^2dx.
    \end{align}
    From \cite[(2.2)]{FZ22}, that is, for $1<p<2$, there exists $c(p)>0$ such that
    \begin{equation}\label{evbcb2i}
    p|x|^{p-2}|y|^2+p(p-2)|\omega|^{p-2}(|x|-|x+y|)^2\geq c(p)\frac{|x|}{|x|+|y|}|x|^{p-2}|y|^2,\quad \forall x\neq 0,\forall y\in\mathbb{R}^N,
    \end{equation}
    we have
    \begin{align*}
    & \left|\frac{\nabla U}{\varepsilon_i}\right|^{p-2}|\nabla \widehat{v}_i|^2
    +(p-2)\left|\frac{\omega_i}{\varepsilon_i}\right|^{p-2}\left(\left|\frac{\nabla U }{\varepsilon_i}+\nabla \widehat{v}_i\right|-\left|\frac{\nabla  U}{\varepsilon_i}\right|\right)^2 \\
    \geq & c(p)\frac{|\nabla U|/\varepsilon_i}{|\nabla U|/\varepsilon_i+|\nabla \widehat{v}_i|}\left|\frac{\nabla  U}{\varepsilon_i}\right|^{p-2}|\nabla \widehat{v}_i|^2,
    \end{align*}
    then,
    \[
    \left| \nabla U \right|^{p-2}|\nabla \widehat{v}_i|^2
    +(p-2)\left| \omega_i \right|^{p-2}\left(\frac{|\nabla (U+v_i)|-|\nabla  U| }{\varepsilon_i}\right)^2
     \geq c(p) \left| \nabla  U \right|^{p-2}|\nabla \widehat{v}_i|^2,\quad \mbox{in}\quad \mathcal{R}_{i,R}.
    \]
    Therefore, combining this bound with \eqref{evbcb2}, we obtain
    \begin{align}\label{evbcb2l}
    & c(p)\int_{\mathcal{R}_{i,R}}|x|^\alpha\left| \nabla  U \right|^{p-2}|\nabla \widehat{v}_i|^2 dx
    +\gamma_0\varepsilon^{p-2}_i\int_{\mathcal{S}_{i,R}}|x|^\alpha|\nabla \widehat{v}_i|^p dx \nonumber\\
    \leq & \int_{B(0,R)\backslash B(0,1/R)}|x|^{\alpha}\bigg[
    |\nabla U|^{p-2}|\nabla \widehat{v}_i|^2
    +(p-2)|\omega_i|^{p-2}\left(\frac{|\nabla (U+v_i)|-|\nabla U|}{\varepsilon_i}\right)^2 \nonumber\\
    & \quad\quad + \gamma_0 \min\{\varepsilon_i^{p-2}|\nabla \widehat{v}_i|^p,|\nabla U|^{p-2}|\nabla \widehat{v}_i|^2\}
    \bigg]dx \nonumber\\
    < & \left[(p^*_{\alpha,\beta}-1)+\tau\right]
    \int_{\mathbb{R}^N}|x|^{\beta}\frac{(U+C_1|v_i|)^{p^*_{\alpha,\beta}}}{U^2+|v_i|^2}|\widehat{v}_i|^2dx.
    \end{align}
    In particular, this implies that
    \begin{align}\label{evbcb2li}
    1= & \varepsilon^{-2}_i\int_{\mathbb{R}^N}|x|^\alpha(|\nabla U|+|\nabla v_i|)^{p-2}|\nabla v_i|^2 dx \nonumber\\
    \leq & C(p)\left[\int_{\mathcal{R}_{i}}|x|^\alpha\left| \nabla  U \right|^{p-2}|\nabla \widehat{v}_i|^2 dx
    +\varepsilon^{p-2}_i\int_{\mathcal{S}_{i}}|x|^\alpha|\nabla \widehat{v}_i|^p dx\right] \nonumber\\
    \leq & C(N,p,\gamma_0)\left[(p^*_{\alpha,\beta}-1)+\tau\right]
    \int_{\mathbb{R}^N}|x|^{\beta}\frac{(U+C_1|v_i|)^{p^*_{\alpha,\beta}}}{U^2+|v_i|^2}|\widehat{v}_i|^2dx.
    \end{align}
    Furthermore, thanks to \eqref{propcetlcpi} in Corollary \ref{propcetlpi}, for $i$ large enough so that $\varepsilon_i$ small we have
    \begin{align}\label{evbcb2lib}
    \int_{\mathbb{R}^N}|x|^{\beta}
    \frac{(U+C_1|v_i|)^{p^*_{\alpha,\beta}}}{U^2+|v_i|^2}|\widehat{v}_i|^2dx
    \leq & C(N,p,C_1)\int_{\mathbb{R}^N}|x|^{\beta}
    (U+|v_i|)^{p^*_{\alpha,\beta}-2}|\widehat{v}_i|^2dx \nonumber\\
    \leq & C(N,p,\alpha,\beta,C_1)\int_{\mathbb{R}^N}|x|^{\alpha}
    (|\nabla U|+|\nabla v_i|)^{p-2}|\nabla \widehat{v}_i|^2dx \nonumber\\
    = & C(N,p,\alpha,\beta,C_1).
    \end{align}
    Hence, combining \eqref{evbcb2l} with \eqref{evbcb2lib}, by the definition of $\mathcal{S}_{i,R}$ we have
    \begin{equation*}
    \begin{split}
    \varepsilon^{-2}_i\int_{\mathcal{S}_{i,R}}|x|^\alpha|\nabla U|^p dx
    \leq \frac{\varepsilon^{-2}_i}{2^p}\int_{\mathcal{S}_{i,R}}|x|^\alpha|\nabla v_i|^p dx
    = \frac{\varepsilon^{p-2}_i}{2^p}\int_{\mathcal{S}_{i,R}}|x|^\alpha|\nabla \widehat{v}_i|^p dx
    \leq C(N,p,\alpha,\beta,C_1),
    \end{split}
    \end{equation*}
    then since
    \[
    0<c(R)\leq |\nabla U|\leq C(R)\quad \mbox{inside}\quad B(0,R)\backslash B(0,1/R),\quad \forall R>1,
    \]
    for some constants $c(R)\leq C(R)$ depending only on $R$, we conclude that
    \begin{equation}\label{csirt0}
    |\mathcal{S}_{i,R}|\to 0\quad \mbox{as}\quad i\to \infty,\quad \forall R>1.
    \end{equation}
    Now, from \eqref{defevh} we have
    \[
    \int_{\mathbb{R}^N}|x|^\alpha(|\nabla U|+\varepsilon_i|\nabla \widehat{v}_i|)^{p-2}|\nabla \widehat{v}_i|^2 dx\leq 1,
    \]
    then according to Lemma \ref{propcetl}, we deduce that $\widehat{v}_i\rightharpoonup \widehat{v}$ in $\mathcal{D}^{1,p}_\alpha(\mathbb{R}^N)$ for some $\widehat{v}\in \mathcal{D}^{1,p}_{\alpha,r}(\mathbb{R}^N)\cap L^2_{\beta,*}(\mathbb{R}^N)$, and
    \begin{equation}\label{csirtc}
    \int_{\mathbb{R}^N}|x|^{\beta}
    \frac{(U+C_1|v_i|)^{p^*_{\alpha,\beta}}}{U^2+|v_i|^2}|\widehat{v}_i|^2dx
    \to \int_{\mathbb{R}^N}|x|^{\beta}U^{p^*_{\alpha,\beta}-2}
    |\widehat{v}|^2dx,
    \end{equation}
    as $i\to \infty$, for any $C_1\geq 0$.
    Also, using \eqref{evbcb2l} and \eqref{evbcb2lib} again we have
    \begin{equation*}
    \int_{\mathcal{R}_{i,R}}|x|^\alpha\left| \nabla  U \right|^{p-2}|\nabla \widehat{v}_i|^2 dx
    \leq C(N,p,\alpha,\beta,C_1),
    \end{equation*}
    therefore \eqref{csirt0} and $\widehat{v}_i\rightharpoonup \widehat{v}$ in $\mathcal{D}^{1,p}_\alpha(\mathbb{R}^N)$ imply that, up to a subsequence,
    \[
    \widehat{v}_i \chi_{\mathcal{R}_{i,R}}\rightharpoonup \widehat{v} \chi_{B(0,R)\backslash B(0,1/R)}\quad \mbox{in}\quad \mathcal{D}^{1,2}_{\alpha,*}(\mathbb{R}^N),\quad \forall R>1.
    \]
    In addition, letting $i\to \infty$ in \eqref{evbcb2li} and \eqref{evbcb2lib}, and using \eqref{csirtc}, we deduce that
    \begin{equation}\label{csirtl1}
    0<c(N,p,\alpha,\beta,C_1,\gamma_0)
    \leq \|\widehat{v}\|_{L^2_{\beta,*}(\mathbb{R}^N)}
    \leq C(N,p,\alpha,\beta,C_1).
    \end{equation}
    Let us write
    \[
    \widehat{v}_i=\widehat{v}+\varphi_i,\quad\mbox{with}\quad \varphi_i:= \widehat{v}_i-\widehat{v},
    \]
    we have
    \[
    \varphi_i \rightharpoonup 0 \quad \mbox{in}\quad   \mathcal{D}^{1,p}_{\alpha}(\mathbb{R}^N)
    \quad \mbox{and}\quad \varphi_i \chi_{\mathcal{R}_{i}}\rightharpoonup 0
    \quad \mbox{locally in} \quad \mathcal{D}^{1,2}_{\alpha,*}(\mathbb{R}^N).
    \]
    We now look at the left side of \eqref{evbcb2}.
    The strong convergence $\equiv v_i\to 0$ in $\mathcal{D}^{1,p}_{\alpha}(\mathbb{R}^N)$ implies that, $|\omega_i|\to |\nabla U|$ a.e. in $\mathbb{R}^N$. Then, let us rewrite
    \begin{align*}
    \left(\frac{|\nabla (U+v_i)|-|\nabla U|}{\varepsilon_i}\right)^2
    = & \left(\left[\int^1_0\frac{\nabla U+ t\nabla v_i}{|\nabla U+ t\nabla v_i|}dt\right]\cdot \nabla \widehat{v}_i\right)^2 \\
    = & \left(\left[\int^1_0\frac{\nabla U+ t\nabla v_i}{|\nabla U+ t\nabla v_i|}dt\right]\cdot \nabla (\widehat{v}+\varphi_i)\right)^2.
    \end{align*}
    Hence, if we set
    \[
    f_{i,1}=\left[\int^1_0\frac{\nabla U+ t\nabla v_i}{|\nabla U+ t\nabla v_i|}dt\right]\cdot \nabla \widehat{v},\quad
    f_{i,2}=\left[\int^1_0\frac{\nabla U+ t\nabla v_i}{|\nabla U+ t\nabla v_i|}dt\right]\cdot \nabla \varphi_i,
    \]
    since $\frac{\nabla U+ t\nabla v_i}{|\nabla U+ t\nabla v_i|}\to \frac{\nabla U}{|\nabla U|}$ a.e., it follows from Lebesgue's dominated convergence theorem that
    \[
    f_{i,1}\to \frac{\nabla U}{|\nabla U|}\cdot\nabla \widehat{v}\quad \mbox{locally in}\quad L^2(\mathbb{R}^N\backslash\{0\}),\quad f_{i,2}\chi_{\mathcal{R}_i}\rightharpoonup 0\quad \mbox{locally in}\quad L^2(\mathbb{R}^N\backslash\{0\}).
    \]
    Thus, the left hand side of \eqref{evbcb2} from below as follows:
    \begin{align}\label{evbcbb2}
    & \int_{\mathcal{R}_{i,R}}|x|^{\alpha}\left[|\nabla U|^{p-2}|\nabla \widehat{v}_i|^2+(p-2)|\omega_i|^{p-2}\left(\frac{|\nabla (U+v_i)|-|\nabla U|}{\varepsilon_i}\right)^2\right]dx \nonumber\\
    = & \int_{\mathcal{R}_{i,R}}|x|^{\alpha}\left[|\nabla U|^{p-2}\left(|\nabla \widehat{v}|^2+2 \nabla \varphi_i\cdot \nabla \widehat{v}\right)+(p-2)|\omega_i|^{p-2}\left(f_{i,1}^2+2f_{i,1}f_{i,2}\right)\right]dx \nonumber\\
    & + \int_{\mathcal{R}_{i,R}}|x|^{\alpha}\left[|\nabla U|^{p-2} |\nabla \varphi_i|^2+(p-2)|\omega_i|^{p-2}f_{i,2}^2\right]dx
    \nonumber\\
    \geq & \int_{\mathcal{R}_{i,R}}|x|^{\alpha}\left[|\nabla U|^{p-2}\left(|\nabla \widehat{v}|^2+2 \nabla \varphi_i\cdot \nabla \widehat{v}\right)+(p-2)|\omega_i|^{p-2}\left(f_{i,1}^2+2f_{i,1}f_{i,2}\right)\right]dx,
    \end{align}
    where the last inequality follows from the nonnegativity of $\left[|\nabla U|^{p-2} |\nabla \varphi_i|^2+(p-2)|\omega_i|^{p-2}f_{i,2}^2\right]$ (thanks to Lemma \ref{lemui1p} and the fact that $f_{i,2}^2\leq |\nabla \varphi_i|^2$). Then, combining the convergence
    \begin{equation*}
    \begin{split}
    & \nabla \varphi_i \chi_{\mathcal{R}_i}\rightharpoonup 0,
    \quad f_{i,1}\to \frac{\nabla U}{|\nabla U|}\cdot\nabla \widehat{v},
    \quad f_{i,2}\chi_{\mathcal{R}_i}\rightharpoonup 0,
    \quad \mbox{locally in}\quad L^2(\mathbb{R}^N\backslash\{0\}),\\
    & |\omega_i|\to |\nabla U|\quad \mbox{a.e.}, \quad |(B(0,R)\backslash B(0,1/R))\backslash\mathcal{R}_{i,R}|=|\mathcal{S}_{i,R}|\to 0,
    \end{split}
    \end{equation*}
    with the fact that
    \[
    |\omega_i|^{p-2}\leq C(p)|\nabla U|^{p-2},
    \]
    by Lebesgue's dominated convergence theorem, we deduce that
    \begin{align*}
    & \lim_{i\to \infty}\int_{\mathcal{R}_{i,R}}|x|^{\alpha}\left[|\nabla U|^{p-2}\left(|\nabla \widehat{v}|^2+2 \nabla \varphi_i\cdot \nabla \widehat{v}\right)+(p-2)|\omega_i|^{p-2}
    \left(f_{i,1}^2+2f_{i,1}f_{i,2}\right)\right]dx \\
    \to & \int_{B(0,R)\backslash B(0,1/R)}|x|^{\alpha}\left[|\nabla U|^{p-2}|\nabla \widehat{v} |^2+(p-2)|\nabla U|^{p-2}\left(\frac{\nabla U\cdot\nabla \widehat{v} }{| \nabla U|}\right)^2\right]dx,
    \end{align*}
    thus from \eqref{evbcbb2} we obtain
    \begin{align*}
    & \liminf_{i\to \infty}\int_{\mathcal{R}_{i,R}}|x|^{\alpha}\left[|\nabla U|^{p-2}|\nabla \widehat{v}_i|^2+(p-2)|\omega_i|^{p-2}\left(\frac{|\nabla (U+v_i)|-|\nabla U|}{\varepsilon_i}\right)^2\right]dx \\
    \geq & \int_{B(0,R)\backslash B(0,1/R)}|x|^{\alpha}\left[|\nabla U|^{p-2}|\nabla \widehat{v} |^2+(p-2)|\nabla U|^{p-2}\left(\frac{\nabla U\cdot\nabla \widehat{v} }{| \nabla U|}\right)^2\right]dx.
    \end{align*}
    Recalling \eqref{evbcb2l} and \eqref{csirtc}, since $R>1$ is arbitrary and the integrand is nonnegative, this proves that
    \begin{align}\label{uinb2p2}
    & \int_{\mathbb{R}^N}|x|^{\alpha}\left[|\nabla U|^{p-2}|\nabla \widehat{v}|^2+(p-2)|\nabla U|^{p-2}\left(\frac{\nabla U\cdot\nabla \widehat{v}}{|\nabla U|}\right)^2\right]dx \nonumber\\
    \leq  &  \left[(p^*_{\alpha,\beta}-1)+\tau\right]
    \int_{\mathbb{R}^N}|x|^{\beta}U^{p^*_{\alpha,\beta}-2}|\widehat{v}|^2dx,
    \end{align}
    The orthogonality of $v_i$ (and also of $\widehat{v}_i$) implies that $\widehat{v}$ also is orthogonal to $T_{U} \mathcal{M}$. Since $\widehat{v}\in L^2_{\beta,*}(\mathbb{R}^N)$, \eqref{csirtl1} and \eqref{uinb2p2} contradict Proposition \ref{propevl}, completing the proof.

    $\bullet$ {\em The case $\frac{2(N+\alpha)}{N+2+\beta}< p<2$} which implies $p^*_{\alpha,\beta}> 2$. If the statement fails, there exists a sequence $0\not\equiv v_i\to 0$ in $\mathcal{D}^{1,p}_{\alpha,r}(\mathbb{R}^N)$, with $v_i$ orthogonal to $T_{U} \mathcal{M}$, such that
    \begin{small}
    \begin{align}\label{evbc2g}
    & \int_{\mathbb{R}^N}|x|^{\alpha}\bigg[
    |\nabla U|^{p-2}|\nabla v_i|^2
    +(p-2)|\omega_i|^{p-2}(|\nabla (U+v_i)|-|\nabla U|)^2
    \nonumber\\ &\quad\quad +\gamma_0 \min\{|\nabla v_i|^p,|\nabla U|^{p-2}|\nabla v_i|^2\}
    \bigg]dx
    \nonumber\\  < &  \left[(p^*_{\alpha,\beta}-1)+\tau\right]
    \int_{\mathbb{R}^N}|x|^{\beta}U^{p^*_{\alpha,\beta}-2}|v_i|^2dx,
    \end{align}
    \end{small}
    where $\omega_i$ corresponds to $v_i$ as in the statement. As in the case $1<p<\frac{2(N+\alpha)}{N+2+\beta}$, we define
    \[
    \varepsilon_i:=\left(\int_{\mathbb{R}^N}|x|^\alpha(|\nabla U|+|\nabla v_i|)^{p-2}|\nabla v_i|^2 dx\right)^{\frac{1}{2}},\quad \widehat{v}_i=\frac{v_i}{\varepsilon_i},
    \]
    and we also have $\varepsilon_i\to 0$ as $i\to \infty$.

    Then, we also split $B(0,R)\backslash B(0,1/R)=\mathcal{R}_{i,R} \cup \mathcal{S}_{i,R}$, \eqref{evbcb2l} and \eqref{evbcb2li} hold also in this case, with the only difference that the last term in both equations now becomes $\left[(p^*_{\alpha,\beta}-1)S_r+ \tau\right]
    \|U\|^{p-p^*_{\alpha,\beta}}_{*}
    \int_{\mathbb{R}^N}|x|^{\beta}U^{p^*_{\alpha,\beta}-2}|v_i|^2dx$.

    We now observe that, by using H\"{o}lder inequality, we have
    \begin{align*}
    \int_{\mathbb{R}^N}|x|^\alpha |\nabla \widehat{v}_i|^p dx
    \leq & \left(\int_{\mathbb{R}^N}|x|^\alpha (|\nabla U|+|\nabla v_i|)^{p-2}|\nabla \widehat{v}_i|^2 dx\right)^{\frac{p}{2}}
    \left(\int_{\mathbb{R}^N}|x|^\alpha (|\nabla U|+|\nabla v_i|)^{p}dx\right)^{1-\frac{p}{2}} \\
    = & \left(\int_{\mathbb{R}^N}|x|^\alpha (|\nabla U|+|\nabla v_i|)^{p}dx\right)^{1-\frac{p}{2}} \\
    \leq & C(p)\left[\left(\int_{\mathbb{R}^N}|x|^\alpha  |\nabla U|^{p}dx\right)^{1-\frac{p}{2}}
    +\varepsilon_i^{\frac{p(2-p)}{2}}\left(\int_{\mathbb{R}^N}|x|^\alpha  |\nabla \widehat{v}_i|^{p}dx\right)^{1-\frac{p}{2}}\right]
    \end{align*}
    from which it follows that
    \begin{equation}\label{vihl}
    \begin{split}
    \int_{\mathbb{R}^N}|x|^\alpha |\nabla \widehat{v}_i|^p dx
    \leq C(N,p,\alpha,\beta).
    \end{split}
    \end{equation}
    Thus, up to a subsequence, $\widehat{v}_i\rightharpoonup \widehat{v}$ in $\mathcal{D}^{1,p}_{\alpha,r}(\mathbb{R}^N)$ and $\widehat{v}_i \to \widehat{v}$ locally in $L^2(\mathbb{R}^N)$. In addition, Then by H\"{o}lder inequality and Sobolev inequality, together with \eqref{vihl}, yield for any $\rho>0$, it holds that
    \begin{align*}
    \int_{\mathbb{R}^N\backslash B(0,\rho)}|x|^\beta U^{p^*_{\alpha,\beta}-2}|\widehat{v}_i|^2 dx
    \leq & \left(\int_{\mathbb{R}^N\backslash B(0,\rho)}|x|^\beta U^{p^*_{\alpha,\beta}}dx\right)^{1-\frac{2}{p^*_{\alpha,\beta}}}
    \left(\int_{\mathbb{R}^N\backslash B(0,\rho)}|x|^\beta |\widehat{v}_i|^{p^*_{\alpha,\beta}}dx\right)^{\frac{2}{p^*_{\alpha,\beta}}} \\
    \leq & \frac{1}{S_r}\left(\int_{\mathbb{R}^N}|x|^\beta U^{p^*_{\alpha,\beta}}dx\right)^{1-\frac{2}{p^*_{\alpha,\beta}}}
    \left(\int_{\mathbb{R}^N}|x|^\alpha |\nabla \widehat{v}_i|^{p}dx\right)^{\frac{2}{p}}.
    \end{align*}
    Combining \eqref{vihl} and the strong convergence $\widehat{v}_i \to \widehat{v}$ locally in $L^2(\mathbb{R}^N)$, we conclude that $\widehat{v}_i\to \widehat{v}$ in $L^2_{\beta,*}(\mathbb{R}^N)$.

    In particular, letting $i\to \infty$ in the analogue of \eqref{evbcb2li} we obtain
    \begin{equation}\label{csirt0l}
    0<c(N,p,\alpha,\beta,C_1,\gamma_0)
    \leq \|\widehat{v}\|_{L^2_{\beta,*}(\mathbb{R}^N)}
    \leq C(N,p,\alpha,\beta,C_1).
    \end{equation}
    Similarly, the analogue of \eqref{evbcb2l} implies that
    \begin{equation}\label{csirt0g}
    |\mathcal{S}_{i,R}|\to 0\quad \mbox{and}\quad \int_{\mathbb{R}^N}|x|^\alpha |\nabla U|^{p-2}|\nabla \widehat{v}_i|^2 dx\leq C(N,p,\alpha,\beta),\quad \forall R>1.
    \end{equation}
    So, it follows from the weak convergence $\widehat{v}_i\rightharpoonup \widehat{v}$ in $\mathcal{D}^{1,p}_{\alpha}(\mathbb{R}^N)$ that, up to a subsequence,
    \[
    \widehat{v}_i \chi_{\mathcal{R}_{i,R}}\rightharpoonup \widehat{v}\chi_{B(0,R)\backslash B(0,1/R)}
    \quad \mbox{locally in} \quad \mathcal{D}^{1,2}_{\alpha,*}(\mathbb{R}^N),\quad \forall R>1.
    \]
    Thanks to this bound, we can split
    \[
    \widehat{v}_i=\widehat{v}+\varphi_i,\quad\mbox{with}\quad \varphi_i:= \widehat{v}_i-\widehat{v},
    \]
    and very same argument as in the case $1<p<\frac{2(N+\alpha)}{2+\beta-\alpha}$ allows us to deduce that
    \begin{align*}
    & \liminf_{i\to \infty}\int_{\mathcal{R}_{i,R}}|x|^{\alpha}\left[|\nabla U|^{p-2}|\nabla \widehat{v}_i|^2+(p-2)|\omega_i|^{p-2}\left(\frac{|\nabla (U+v_i)|-|\nabla U|}{\varepsilon_i}\right)^2\right]dx \\
     \geq & \int_{B(0,R)\backslash B(0,1/R)}|x|^{\alpha}\left[|\nabla U|^{p-2}|\nabla \widehat{v} |^2+(p-2)|\nabla U|^{p-2}\left(\frac{\nabla U\cdot\nabla \widehat{v} }{| \nabla U|}\right)^2\right]dx.
    \end{align*}
    Recalling \eqref{evbc2g}, since $R>1$ is arbitrary and the integrands above are nonnegative, this proves that \eqref{uinb2p2} holds, a contradiction to Proposition \ref{propevl} since $\widehat{v}$ is orthogonal to $T_{U} \mathcal{M}$ (being the strong $L^2_{\beta,*}(\mathbb{R}^N)$ limit of $\widehat{v}_i$).
    \end{proof}

    The main ingredient in the proof of Theorem \ref{thmprtp2} is contained in the lemma below, where the behavior near $\mathcal{M}$ is studied.

    \begin{lemma}\label{lemma:rtnm2b2}
    Suppose $1<p<2\leq N$.
    There exists a small constant $\rho>0$ such that for any sequence $\{u_n\}\subset \mathcal{D}^{1,p}_{\alpha,r}(\mathbb{R}^N)\backslash \mathcal{M}$ satisfying $\inf_n\|u_n\|>0$ and ${\rm dist}(u_n,\mathcal{M})\to 0$, it holds that
    \begin{equation}\label{rtnmb2}
    \liminf_{n\to\infty}
    \frac{\|u_n\|^p- S_r\|u_n\|_*^p}
    {{\rm dist}(u_n,\mathcal{M})^2}
    \geq \rho.
    \end{equation}
    \end{lemma}

    \begin{proof}
    Let $d_n:={\rm dist}(u_n,\mathcal{M})=\inf_{c\in\mathbb{R}, \lambda>0}\|u_n-cU_\lambda\|_{\mathcal{D}^{1,p}_{\alpha}(\mathbb{R}^N)}\to 0$ as $n\to \infty$. We know that for each $u_n\in \mathcal{D}^{1,p}_{\alpha,r}(\mathbb{R}^N)$, there exist $c_n\in\mathbb{R}$ and $\lambda_n>0$ such that $d_n=\|u_n-c_nU_{\lambda_n}\|_{\mathcal{D}^{1,p}_{\alpha}(\mathbb{R}^N)}$. 
    In fact, since $1< p<2$, for each fixed $n$, from Lemma \ref{lemui1p}, we obtain that for any $0<\kappa<1$, there exists a constant $\mathcal{C}_1=\mathcal{C}_1(p,\kappa)>0$ such that
    \begin{small}
    \begin{align}\label{ikeda2}
    \|u_n-cU_\lambda\|^p
    = & \int_{\mathbb{R}^N}|x|^{\alpha}|\nabla u_n-c\nabla U_\lambda|^p dx\nonumber\\
    \geq & \int_{\mathbb{R}^N}|x|^{\alpha}|\nabla u_n|^p dx
    -pc\int_{\mathbb{R}^N}|x|^{\alpha}|\nabla u_n|^{p-2}  \nabla u_n\cdot \nabla U_\lambda dx  \nonumber\\ &
    +\mathcal{C}_1|c|^{2}\int_{\mathbb{R}^N}|x|^{\alpha}\min\{|c|^{p-2}|\nabla U_\lambda|^p,|\nabla u_n|^{p-2}|\nabla U_\lambda|\} dx
    \nonumber\\ &
    +\frac{(1-\kappa)p}{2}c^2 \int_{\mathbb{R}^N}|x|^{\alpha} |\nabla u_n|^{p-2}  |\nabla U_\lambda|^2dx \nonumber\\
    & +\frac{(1-\kappa)p(p-2)}{2}\int_{\mathbb{R}^N}|x|^{\alpha}|\omega(\nabla u_n, \nabla u_n-c \nabla U_{\lambda})|^{p-2}(|c \nabla U_{\lambda}|-|\nabla u_n|)^2  dx \nonumber\\
    \geq & \int_{\mathbb{R}^N}|x|^{\alpha}|\nabla u_n|^p dx
    -pc\int_{\mathbb{R}^N}|x|^{\alpha}|\nabla u_n|^{p-2}  \nabla u_n\cdot \nabla U_\lambda dx  \nonumber\\ &
    +\mathcal{C}_1|c|^{2}\int_{\mathbb{R}^N}|x|^{\alpha}\min\{|c|^{p-2}|\nabla U_\lambda|^p,|\nabla u_n|^{p-2}|\nabla U_\lambda|\} dx
    \nonumber\\
    \geq & \|u_n\|^p-p|c|\|U\|\|u_n\|^{p-1}\nonumber\\ &
    +\mathcal{C}_1|c|^{2}\int_{\mathbb{R}^N}|x|^{\alpha}\min\{|c|^{p-2}|\nabla U_\lambda|^p,|\nabla u_n|^{p-2}|\nabla U_\lambda|\} dx,
    \end{align}
    \end{small}
    where $\omega:\mathbb{R}^{2N}\to \mathbb{R}^N$ corresponds to $\nabla u_n$ and $\nabla u_n-c \nabla U_{\lambda}$ as in Lemma \ref{lemui1p} for the case $1<p<2$.
    Thus the minimizing sequence of $d_n$, say $\{c_{n,m},\lambda_{n,m}\}$, must satisfying $1/C\leq |c_{n,m}|\leq C$ for some $C>0$,  which means $\{c_{n,m}\}$ is bounded.
    On the other hand, taking the same steps as those in Lemma \ref{lemma:rtnm2b}, we deduce that
    \[\left|\int_{\mathbb{R}^N}|x|^{\alpha}|\nabla u_n|^{p-2}  \nabla u_n\cdot \nabla U_\lambda dx\right| \to 0\quad \mbox{as}\quad \lambda\to 0,\]
    and
    \[\left|\int_{\mathbb{R}^N}|x|^{\alpha}|\nabla u_n|^{p-2}  \nabla u_n\cdot \nabla U_\lambda dx\right| \to 0\quad \mbox{as}\quad \lambda\to +\infty.\]
    It follows from (\ref{ikeda}) and $d_n\to 0$, $\inf_n\|u_n\|>0$ that the minimizing sequence $\{c_{n,m},\lambda_{n,m}\}$ must satisfying $1/C\leq |\lambda_{n,m}|\leq C$ for some $C>1$,  which means $\{\lambda_{n,m}\}$ is bounded. Thus for each $u_n\in  \mathcal{D}^{1,p}_{\alpha,r}(\mathbb{R}^N)\backslash \mathcal{M}$, $d_n$ can also be attained by some $c_n\in\mathbb{R}$ and $\lambda_n>0$.

    Since $\mathcal{M}$ is two-dimensional manifold embedded in $\mathcal{D}^{1,p}_{\alpha,r}(\mathbb{R}^N)$, that is
    \[
    (c,\lambda)\in\mathbb{R}\times\mathbb{R}_+\to cU_\lambda\in \mathcal{D}^{1,p}_{\alpha,r}(\mathbb{R}^N),
    \]
    then from Proposition \ref{propev}, the tangential space at $(c_n,\lambda_n)$ is given by
    \[
    T_{c_n U_{\lambda_n}}\mathcal{M}={\rm Span}\left\{U_{\lambda_n}, \quad \lambda_n^{\frac{N-p+\alpha}{p}}W_0(\lambda_n x),\quad \lambda_n^{\frac{N-p+\alpha}{p}}W_{k,i}(\lambda_n x), i=1,\ldots, M_k\right\},
    \]
    if (\ref{npkev}) holds, otherwise
    \[
    T_{c_n U_{\lambda_n}}\mathcal{M}={\rm Span}\left\{U_{\lambda_n}, \quad \lambda_n^{\frac{N-p+\alpha}{p}}W_0(\lambda_n x)\right\}.
    \]
    Anyway we must have that $(u_n-c_n U_{\lambda_n})$ is perpendicular to $T_{c_n U_{\lambda_n}}\mathcal{M}$, particularly
    \begin{equation*}
    \int_{\mathbb{R}^N}|x|^\alpha |\nabla U_{\lambda_n}|^{p-2}\nabla U_{\lambda_n}\cdot \nabla (u_n-c_n U_{\lambda_n}) dx=0.
    \end{equation*}
    Let
    \begin{equation}\label{defunwn2}
    u_n=c_n U_{\lambda_n}+d_n w_n,
    \end{equation}
     then $w_n$ is perpendicular to $T_{c_n U_{\lambda_n}}\mathcal{M}$, we have
    \begin{equation*}
    \|w_n\|=1\quad \mbox{and}\quad \int_{\mathbb{R}^N}|x|^\alpha |\nabla U_{\lambda_n}|^{p-2}\nabla U_{\lambda_n}\cdot \nabla w_n dx=0.
    \end{equation*}
    Since $1<p<2$, from Lemma \ref{lemui1p} we obtain that for any $\kappa>0$, there exists a constant $\mathcal{C}_2=\mathcal{C}_2(p,\kappa)>0$ such that
    \begin{align}\label{epknug2}
    \int_{\mathbb{R}^N}|x|^{\alpha}|\nabla u_n|^p dx
    = & \int_{\mathbb{R}^N}|x|^{\alpha}|c_n \nabla U_{\lambda_n}+d_n \nabla w_n|^p dx \nonumber\\
    \geq & |c_n|^{p}\int_{\mathbb{R}^N}|x|^{\alpha}|\nabla U_{\lambda_n}|^p dx
    +p|c_n|^{p-2}c_nd_n \int_{\mathbb{R}^N}|x|^{\alpha}|\nabla U_{\lambda_n}|^{p-2}  \nabla U_{\lambda_n}\cdot \nabla w_n dx  \nonumber\\
    & +\frac{(1-\kappa)p}{2} |c_n|^{p-2}d_n^2\int_{\mathbb{R}^N}|x|^{\alpha} |\nabla U_{\lambda_n}|^{p-2}  |\nabla w_n|^2dx \nonumber\\
    & +\frac{(1-\kappa)p(p-2)}{2}\int_{\mathbb{R}^N}|x|^{\alpha}|\omega(c_n \nabla U_{\lambda_n},\nabla u_n)|^{p-2}(|c_n \nabla U_{\lambda_n}|-|\nabla u_n|)^2  dx \nonumber\\
    & +\mathcal{C}_2d_n^{2} \int_{\mathbb{R}^N}|x|^{\alpha}\min\{d_n^{p-2}|\nabla w_n|^{p}, |c_n \nabla U_{\lambda_n}|^{p-2}|\nabla w_n|^{2}\} dx  \nonumber\\
    = & |c_n|^{p}\|U\|^p
    +\frac{(1-\kappa)p}{2} |c_n|^{p-2}d_n^2\int_{\mathbb{R}^N}|x|^{\alpha} |\nabla U_{\lambda_n}|^{p-2}  |\nabla w_n|^2dx \nonumber\\
    & +\frac{(1-\kappa)p(p-2)}{2}\int_{\mathbb{R}^N}|x|^{\alpha}|\omega(c_n \nabla U_{\lambda_n},\nabla u_n)|^{p-2}(|c_n \nabla U_{\lambda_n}|-|\nabla u_n|)^2  dx\nonumber\\
    & +\mathcal{C}_2d_n^{2} \int_{\mathbb{R}^N}|x|^{\alpha}\min\{d_n^{p-2}|\nabla w_n|^{p}, |c_n \nabla U_{\lambda_n}|^{p-2}|\nabla w_n|^{2}\} dx,
    \end{align}
    where $\omega:\mathbb{R}^{2N}\to \mathbb{R}^N$ corresponds to $c_n \nabla U_{\lambda_n}$ and $u_n$ as in Lemma \ref{lemui1p}. Then we consider the following two cases:

    $\bullet$ {\em The case $1<p\leq\frac{2(N+\alpha)}{N+2+\beta}$} which implies $p^*_{\alpha,\beta}\leq2$.

    From Lemma \ref{lemui1p*l}, for any $\kappa>0$ and $C_1>0$, there exists a constant $\mathcal{C}_1=\mathcal{C}_1(p^*_{\alpha,\beta},\kappa,C_1)>0$ such that
    \begin{align*}
    & \int_{\mathbb{R}^N}|x|^{\beta}|u_n|^{p^*_{\alpha,\beta}} dx \\
    = & \int_{\mathbb{R}^N}|x|^{\beta}|c_n U_{\lambda_n}+d_nw_n|^{p^*_{\alpha,\beta}}  dx\\
    \leq & |c_n|^{p^*_{\alpha,\beta}}\int_{\mathbb{R}^N}|x|^{\beta}U_{\lambda_n}^{p^*_{\alpha,\beta}} dx
    +p^*_{\alpha} |c_n|^{p^*_{\alpha,\beta}-2}c_n d_n \int_{\mathbb{R}^N}|x|^{\beta}U_{\lambda_n}^{p^*_{\alpha,\beta}-1}w_n dx  \\
    & +\left(\frac{p^*_{\alpha,\beta}(p^*_{\alpha,\beta}-1)}{2}+\kappa\right)d_n^2
    \int_{\mathbb{R}^N}|x|^{\beta}\frac{(|c_n U_{\lambda_n}|+\mathcal{C}_1|d_nw_n|)^{p^*_{\alpha,\beta}}}{|c_n U_{\lambda_n}|^2+|d_nw_n|^2}w_n^2 dx \\
    = & |c_n|^{p^*_{\alpha,\beta}}\|U\|^p
    +\left(\frac{p^*_{\alpha,\beta}(p^*_{\alpha,\beta}-1)}{2}+\kappa\right)d_n^2
    \int_{\mathbb{R}^N}|x|^{\beta}\frac{(|c_n U_{\lambda_n}|+\mathcal{C}_1|d_nw_n|)^{p^*_{\alpha,\beta}}}{|c_n U_{\lambda_n}|^2+|d_nw_n|^2}w_n^2 dx,
    \end{align*}
    since
    \begin{equation*}
    \int_{\mathbb{R}^N}|x|^{\beta}U_{\lambda_n}^{p^*_{\alpha,\beta}-1}w_n dx=\int_{\mathbb{R}^N}|x|^\alpha |\nabla U_{\lambda_n}|^{p-2}\nabla U_{\lambda_n}\cdot \nabla w_n dx=0,
    \end{equation*}
    and
    \begin{equation*}
    \int_{\mathbb{R}^N}|x|^{\beta}U_{\lambda_n}^{p^*_{\alpha,\beta}} dx=\int_{\mathbb{R}^N}|x|^\alpha |\nabla U_{\lambda_n}|^{p}dx=\|U\|^p.
    \end{equation*}
    Thus, by the concavity of $t\mapsto t^{\frac{p}{p^*_{\alpha,\beta}}}$, we have
    \begin{align}\label{epkeyiyxbb2}
    \left(\int_{\mathbb{R}^N}|x|^{\beta}|u_n|^{p^*_{\alpha,\beta}} dx\right)^{\frac{p}{p^*_{\alpha,\beta}}}
    \leq  &  |c_n|^p\|U\|^{\frac{p^2}{p^*_{\alpha,\beta}}}
    +\frac{p}{p^*_{\alpha,\beta}}\left(\frac{p^*_{\alpha,\beta}(p^*_{\alpha,\beta}-1)}{2}+\kappa\right)d_n^2
    \|U\|^{\frac{p^2}{p^*_{\alpha,\beta}}-p}
    \nonumber\\ & \quad\times\int_{\mathbb{R}^N}|x|^{\beta}\frac{(|c_n U_{\lambda_n}|+\mathcal{C}_1|d_nw_n|)^{p^*_{\alpha,\beta}}}{|c_n U_{\lambda_n}|^2+|d_nw_n|^2}w_n^2 dx.
    \end{align}
    Therefore, as $d_n\to 0$, combining \eqref{epkeyiyxbb2} with \eqref{epknug2}, it follows from Lemma \ref{lemsgap} that, by choosing $\kappa>0$ small enough,
    \begin{small}
    \begin{align*}
    & \int_{\mathbb{R}^N}|x|^{\alpha}|\nabla u_n|^p dx- S_r\left(\int_{\mathbb{R}^N}|x|^{\beta}|u_n|^{p^*_{\alpha,\beta}} dx\right)^{\frac{p}{p^*_{\alpha,\beta}}} \\
    \geq & |c_n|^{p}\|U\|^p
    +\frac{(1-\kappa)p}{2}d_n^2\int_{\mathbb{R}^N}|x|^{\alpha} |\nabla c_n U_{\lambda_n}|^{p-2}  |\nabla w_n|^2dx \\
    & +\frac{(1-\kappa)p(p-2)}{2}\int_{\mathbb{R}^N}|x|^{\alpha}|\omega(c_n \nabla U_{\lambda_n},\nabla u_n)|^{p-2}(|c_n \nabla U_{\lambda_n}|-|\nabla u_n|)^2  dx\\
    & +\mathcal{C}_2d_n^{2} \int_{\mathbb{R}^N}|x|^{\alpha}\min\{d_n^{p-2}|\nabla w_n|^{p}, |c_n \nabla U_{\lambda_n}|^{p-2}|\nabla w_n|^{2}\} dx \\
    & -S_r\Bigg\{|c_n|^p\|U\|^{\frac{p^2}{p^*_{\alpha,\beta}}}  \\ & \quad \quad +\left(\frac{p(p^*_{\alpha,\beta}-1)}{2}+\frac{p\kappa}{p^*_{\alpha,\beta}}\right)d_n^2
    \|U\|^{\frac{p^2}{p^*_{\alpha,\beta}}-p}\int_{\mathbb{R}^N}|x|^{\beta}\frac{(|c_n U_{\lambda_n}|+\mathcal{C}_1|d_nw_n|)^{p^*_{\alpha,\beta}}}{|c_n U_{\lambda_n}|^2+|d_nw_n|^2}w_n^2 dx\Bigg\} \\
    \geq  & \frac{(1-\kappa)p}{2} d_n^2\int_{\mathbb{R}^N}|x|^{\alpha} |\nabla c_nU_{\lambda_n}|^{p-2}  |\nabla w_n|^2dx \\
    & +\frac{(1-\kappa)p(p-2)}{2}\int_{\mathbb{R}^N}|x|^{\alpha}|\omega(c_n \nabla U_{\lambda_n},\nabla u_n)|^{p-2}(|c_n \nabla U_{\lambda_n}|-|\nabla u_n|)^2  dx\\
    & +\mathcal{C}_2d_n^{2} \int_{\mathbb{R}^N}|x|^{\alpha}\min\{d_n^{p-2}|\nabla w_n|^{p}, |c_n \nabla U_{\lambda_n}|^{p-2}|\nabla w_n|^{2}\} dx \\
    & -\left(\frac{p(p^*_{\alpha,\beta}-1)}{2}+\frac{p\kappa}{p^*_{\alpha,\beta}}\right)d_n^2
    \int_{\mathbb{R}^N}|x|^{\beta}\frac{(|c_n U_{\lambda_n}|+\mathcal{C}_1|d_nw_n|)^{p^*_{\alpha,\beta}}}{|c_n U_{\lambda_n}|^2+|d_nw_n|^2}w_n^2 dx,
    \end{align*}
    \end{small}
    since $\|U\|^p=\|U\|_*^{p^*_{\alpha,\beta}}=S_r^{\frac{p^*_{\alpha,\beta}}{p^*_{\alpha,\beta}-p}}$ and $\|U\|^p=S_r\|U\|_*^p$. Lemma \ref{lemsgap2} allows us to reabsorb the last term above: more precisely, we have
    \begin{small}
    \begin{align*}
    & \int_{\mathbb{R}^N}|x|^{\alpha}|\nabla u_n|^p dx- S_r\left(\int_{\mathbb{R}^N}|x|^{\beta}|u_n|^{p^*_{\alpha,\beta}} dx\right)^{\frac{p}{p^*_{\alpha,\beta}}} \\
    \geq & pd_n^2\left(\frac{(1-\kappa)}{2} - \frac{(p^*_{\alpha,\beta}-1)+\frac{2}{p^*_{\alpha,\beta}}\kappa}{2(p^*_{\alpha,\beta}-1)
    +2\tau }\right) \\
    & \quad \times
    \int_{\mathbb{R}^N}|x|^{\alpha} \left[|\nabla c_nU_{\lambda_n}|^{p-2}  |\nabla w_n|^2
    +(p-2)|\omega(c_n \nabla U_{\lambda_n},\nabla u_n)|^{p-2}\left(\frac{|c_n \nabla U_{\lambda_n}|-|\nabla u_n|}{d_n}\right)^2\right]dx \\
    & + d_n^{2}\left(\mathcal{C}_2 - \gamma_0\frac{p\left[(p^*_{\alpha,\beta}-1)+\frac{2}{p^*_{\alpha,\beta}}\kappa\right]}
    {2(p^*_{\alpha,\beta}-1)+2\tau }\right)
    \int_{\mathbb{R}^N}|x|^{\alpha} \min\{d_n^{p-2}|\nabla w_n|^{p}, |c_n \nabla U_{\lambda_n}|^{p-2}|\nabla w_n|^{2}\} dx.
    \end{align*}
    \end{small}
    Now, let us recall the definition of $\omega$, as stated in Lemma \ref{lemui1p}, we have
    \[
    |\nabla c_nU_{\lambda_n}|^{p-2}  |\nabla w_n|^2
    +(p-2)|\omega(c_n \nabla U_{\lambda_n},\nabla u_n)|^{p-2}\left(\frac{|c_n \nabla U_{\lambda_n}|-|\nabla u_n|}{d_n}\right)^2\geq 0,
    \]
    then choosing $\kappa>0$ small enough such that
    \[
    \frac{(1-\kappa)}{2} - \frac{(p^*_{\alpha,\beta}-1)+\frac{2}{p^*_{\alpha,\beta}}\kappa}{2(p^*_{\alpha,\beta}-1)
    +2\tau }\geq 0,
    \]
    and then $\gamma_0>0$ small enough such that
    \[
    \frac{\mathcal{C}_2}{2}\geq \gamma_0\frac{p\left[(p^*_{\alpha,\beta}-1)
    +\frac{2}{p^*_{\alpha,\beta}}\kappa\right]}
    {2(p^*_{\alpha,\beta}-1)+2\tau },
    \]
    we eventually arrive at
    \begin{small}
    \begin{align}\label{emn21}
    & \int_{\mathbb{R}^N}|x|^{\alpha}|\nabla u_n|^p dx- S_r\left(\int_{\mathbb{R}^N}|x|^{\beta}|u_n|^{p^*_{\alpha,\beta}} dx\right)^{\frac{p}{p^*_{\alpha,\beta}}}
    \nonumber\\ \geq & \frac{\mathcal{C}_2}{2}d_n^{2}\int_{\mathbb{R}^N}|x|^{\alpha} \min\{d_n^{p-2}|\nabla w_n|^{p}, |c_n \nabla U_{\lambda_n}|^{p-2}|\nabla w_n|^{2}\} dx.
    \end{align}
    \end{small}
    Observe that, since $1<p<2$, it follows by H\"{o}lder inequality that
    \begin{align*}
    \left(\int_{\{d_n|\nabla w_n|\geq|c_n \nabla U_{\lambda_n}|\}}|x|^{\alpha}  |\nabla w_n|^{p} dx\right)^{\frac{2}{p}}
    \leq & \left(\int_{\{d_n|\nabla w_n|\geq|c_n \nabla U_{\lambda_n}|\}}|x|^{\alpha}  |\nabla U_{\lambda_n}|^{p} dx\right)^{\frac{2}{p}-1}  \\
    & \quad \times
    \int_{\{d_n|\nabla w_n|\geq|c_n \nabla U_{\lambda_n}|\}}|x|^{\alpha}  |\nabla U_{\lambda_n}|^{p-2}|\nabla w_n|^{2} dx  \\
    \leq & S_r^{\frac{p^*_{\alpha,\beta}(\frac{2}{p}-1)}{p^*_{\alpha,\beta}-p}} \int_{\{d_n|\nabla w_n|\geq|c_n \nabla U_{\lambda_n}|\}}|x|^{\alpha}  |\nabla U_{\lambda_n}|^{p-2}|\nabla w_n|^{2} dx,
    \end{align*}
    then we obtain
    \begin{align}\label{emn21b}
    & \int_{\mathbb{R}^N}|x|^{\alpha} \min\{d_n^{p-2}|\nabla w_n|^{p}, |c_n \nabla U_{\lambda_n}|^{p-2}|\nabla w_n|^{2}\} dx
    \nonumber\\
    = & d_n^{p-2}\int_{\{d_n|\nabla w_n|<|c_n \nabla U_{\lambda_n}|\}}|x|^{\alpha}  |\nabla w_n|^{p} dx
    +\int_{\{d_n|\nabla w_n|\geq |c_n \nabla U_{\lambda_n}|\}}|x|^{\alpha}   |c_n \nabla U_{\lambda_n}|^{p-2}|\nabla w_n|^{2}dx
    \nonumber\\
    \geq & d_n^{p-2}\int_{\{d_n|\nabla w_n|< |c_n \nabla U_{\lambda_n}|\}}|x|^{\alpha}  |\nabla w_n|^{p} dx
    + c\left(\int_{\{d_n|\nabla w_n|\geq|c_n \nabla U_{\lambda_n}|\}}|x|^{\alpha}  |\nabla w_n|^{p} dx\right)^{\frac{2}{p}}
    \nonumber\\
    \geq & c\left(\int_{\mathbb{R}^N}|x|^{\alpha}  |\nabla w_n|^{p} dx\right)^{\frac{2}{p}}=c,
    \end{align}
    for some constant $c>0$.
    The conclusion (\ref{rtnmb2}) follows immediately from \eqref{emn21} and \eqref{emn21b}.

    $\bullet$ {\em The case $\frac{2(N+\alpha)}{N+2+\beta}< p<2$} which implies $p^*_{\alpha,\beta}> 2$.

    The proof is very similar to the previous case, with very small changes. From Lemma \ref{lemui1p*l}, we have that for any $\kappa>0$, there exists a constant $\mathcal{C}_1=\mathcal{C}_1(p^*_{\alpha,\beta},\kappa)>0$ such that
    \begin{align*}
    & \int_{\mathbb{R}^N}|x|^{\beta}|u_n|^{p^*_{\alpha,\beta}} dx \\
    \leq & |c_n|^{p^*_{\alpha,\beta}}\int_{\mathbb{R}^N}|x|^{\beta}U_{\lambda_n}^{p^*_{\alpha,\beta}} dx
    +|c_n|^{p^*_{\alpha,\beta}-2}c_n p^*_{\alpha} d_n \int_{\mathbb{R}^N}|x|^{\beta}U_{\lambda_n}^{p^*_{\alpha,\beta}-1}w_n dx  \\
    & +\left(\frac{p^*_{\alpha,\beta}(p^*_{\alpha,\beta}-1)}{2}+\kappa\right)|c_n|^{p^*_{\alpha,\beta}-2} d_n^2
    \int_{\mathbb{R}^N}|x|^{\beta}U_{\lambda_n}^{p^*_{\alpha,\beta}-2}w_n^2 dx \\
    & +\mathcal{C}_2d_n^{p^*_{\alpha,\beta}}\int_{\mathbb{R}^N}|x|^{\beta}|w_n|^{p^*_{\alpha,\beta}} dx \\
    = & |c_n|^{p^*_{\alpha,\beta}}\|U\|^p
    +\left(\frac{p^*_{\alpha,\beta}(p^*_{\alpha,\beta}-1)}{2}+\kappa\right)|c_n|^{p^*_{\alpha,\beta}-2} d_n^2
    \int_{\mathbb{R}^N}|x|^{\beta}U_{\lambda_n}^{p^*_{\alpha,\beta}-2}w_n^2 dx
    + o(d_n^2).
    \end{align*}
    Then by the concavity of $t\mapsto t^{\frac{p}{p^*_{\alpha,\beta}}}$, we have
    \begin{align}\label{epkeyiyxbbg2}
    \left(\int_{\mathbb{R}^N}|x|^{\beta}|u_n|^{p^*_{\alpha,\beta}} dx\right)^{\frac{p}{p^*_{\alpha,\beta}}}
    \leq  &  |c_n|^p\|U\|^{\frac{p^2}{p^*_{\alpha,\beta}}}
    + o(d_n^2)
    + \frac{p|c_n|^{p^*_{\alpha,\beta}-2} d_n^2}{p^*_{\alpha,\beta}}\left(\frac{p^*_{\alpha,\beta}(p^*_{\alpha,\beta}-1)}{2}+\kappa\right)
    \nonumber\\ &\quad \times \|U\|^{\frac{p^2}{p^*_{\alpha,\beta}}-p}
    \int_{\mathbb{R}^N}|x|^{\beta}U_{\lambda_n}^{p^*_{\alpha,\beta}-2}w_n^2 dx.
    \end{align}
    Hence, arguing as in the case $1<p<\frac{2(N+\alpha)}{N+2+\beta}$,
    Therefore, as $d_n\to 0$, combining \eqref{epknug2} with \eqref{epkeyiyxbbg2}, it follows from Lemma \ref{lemsgap} that, by choosing $\kappa>0$ small enough,
    \begin{small}
    \begin{align*}
    & \int_{\mathbb{R}^N}|x|^{\alpha}|\nabla u_n|^p dx- S_r\left(\int_{\mathbb{R}^N}|x|^{\beta}|u_n|^{p^*_{\alpha,\beta}} dx\right)^{\frac{p}{p^*_{\alpha,\beta}}} \\
    \geq & |c_n|^{p}\|U\|^p
    +\frac{(1-\kappa)p}{2} |c_n|^{p-2}d_n^2\int_{\mathbb{R}^N}|x|^{\alpha} |\nabla U_{\lambda_n}|^{p-2}  |\nabla w_n|^2dx \\
    & +\frac{(1-\kappa)p(p-2)}{2}\int_{\mathbb{R}^N}|x|^{\alpha}|\omega(c_n \nabla U_{\lambda_n},\nabla u_n)|^{p-2}(|c_n \nabla U_{\lambda_n}|-|\nabla u_n|)^2  dx\\
    & +\mathcal{C}_2d_n^{2} \int_{\mathbb{R}^N}|x|^{\alpha}\min\{d_n^{p-2}|\nabla w_n|^{p}, |c_n \nabla U_{\lambda_n}|^{p-2}|\nabla w_n|^{2}\} dx \\
    & -S_r\Bigg\{|c_n|^p\|U\|^{\frac{p^2}{p^*_{\alpha,\beta}}}
    + o(d_n^2) \\
    & \quad \quad+ \frac{p|c_n|^{p^*_{\alpha,\beta}-2} d_n^2}{p^*_{\alpha,\beta}}\left(\frac{p^*_{\alpha,\beta}(p^*_{\alpha,\beta}-1)}{2}+\kappa\right)
    \|U\|^{\frac{p^2}{p^*_{\alpha,\beta}}-p}\int_{\mathbb{R}^N}|x|^{\beta}U_{\lambda_n}^{p^*_{\alpha,\beta}-2}w_n^2 dx
    \Bigg\} \\
    \geq  & \frac{(1-\kappa)p}{2} |c_n|^{p-2}d_n^2\int_{\mathbb{R}^N}|x|^{\alpha} |\nabla U_{\lambda_n}|^{p-2}  |\nabla w_n|^2dx \\
    & +\frac{(1-\kappa)p(p-2)}{2}\int_{\mathbb{R}^N}|x|^{\alpha}|\omega(c_n \nabla U_{\lambda_n},\nabla u_n)|^{p-2}(|c_n \nabla U_{\lambda_n}|-|\nabla u_n|)^2  dx\\
    & +\mathcal{C}_2d_n^{2} \int_{\mathbb{R}^N}|x|^{\alpha}\min\{d_n^{p-2}|\nabla w_n|^{p}, |c_n \nabla U_{\lambda_n}|^{p-2}|\nabla w_n|^{2}\} dx \\
    & -\left(\frac{p(p^*_{\alpha,\beta}-1)}{2}+\frac{p\kappa}{p^*_{\alpha,\beta}}\right)d_n^2
    \int_{\mathbb{R}^N}|x|^{\beta}U_{\lambda_n}^{p^*_{\alpha,\beta}-2}w_n^2 dx
    - o(d_n^2) .
    \end{align*}
    \end{small}
    Lemma \ref{lemsgap2} allows us to reabsorb the last term above: more precisely, we have
    \begin{small}
    \begin{align*}
    & \int_{\mathbb{R}^N}|x|^{\alpha}|\nabla u_n|^p dx- S_r\left(\int_{\mathbb{R}^N}|x|^{\beta}|u_n|^{p^*_{\alpha,\beta}} dx\right)^{\frac{p}{p^*_{\alpha,\beta}}} \\
    \geq & pd_n^2\left(\frac{(1-\kappa)}{2} - \frac{(p^*_{\alpha,\beta}-1)+\frac{2}{p^*_{\alpha,\beta}}\kappa}{2(p^*_{\alpha,\beta}-1)
    +2\tau }\right) \\
    & \quad \times
    \int_{\mathbb{R}^N}|x|^{\alpha} \left[|\nabla U_{\lambda_n}|^{p-2}  |\nabla w_n|^2
    +(p-2)|\omega(c_n \nabla U_{\lambda_n},\nabla u_n)|^{p-2}\left(\frac{|c_n \nabla U_{\lambda_n}|-|\nabla u_n|}{d_n}\right)^2\right]dx \\
    & + d_n^{2}\left(\mathcal{C}_2 - \gamma_0\frac{p\left[(p^*_{\alpha,\beta}-1)+\frac{2}{p^*_{\alpha,\beta}}\kappa\right]}
    {2(p^*_{\alpha,\beta}-1)+2\tau }\right)
    \int_{\mathbb{R}^N}|x|^{\alpha}\min\{d_n^{p-2}|\nabla w_n|^{p}, |c_n \nabla U_{\lambda_n}|^{p-2}|\nabla w_n|^{2}\} dx
    \\ & - o(d_n^2) .
    \end{align*}
    \end{small}
    Now, let us recall the definition of $\omega$, as stated in Lemma \ref{lemui1p}, we have
    \[
    |\nabla c_nU_{\lambda_n}|^{p-2}  |\nabla w_n|^2
    +(p-2)|\omega(c_n \nabla U_{\lambda_n},\nabla u_n)|^{p-2}\left(\frac{|c_n \nabla U_{\lambda_n}|-|\nabla u_n|}{d_n}\right)^2\geq 0,
    \]
    then  choosing $\kappa>0$ small enough such that
    \[
    \frac{(1-\kappa)}{2} - \frac{(p^*_{\alpha,\beta}-1)+\frac{2}{p^*_{\alpha,\beta}}\kappa}{2(p^*_{\alpha,\beta}-1)
    +2\tau }\geq 0,
    \]
    and then choosing $\gamma_0>0$ small enough such that
    \[
    \frac{\mathcal{C}_2}{2}\geq \gamma_0\frac{p\left[(p^*_{\alpha,\beta}-1)+\frac{2}{p^*_{\alpha,\beta}}\kappa\right]}
    {2(p^*_{\alpha,\beta}-1)+2\tau }.
    \]
    From \eqref{emn21b}, we eventually arrive at
    \begin{small}
    \begin{align}\label{emn212}
    & \int_{\mathbb{R}^N}|x|^{\alpha}|\nabla u_n|^p dx- S_r\left(\int_{\mathbb{R}^N}|x|^{\beta}|u_n|^{p^*_{\alpha,\beta}} dx\right)^{\frac{p}{p^*_{\alpha,\beta}}}
    \nonumber\\
    \geq & \frac{\mathcal{C}_2}{2}d_n^{2}\int_{\mathbb{R}^N}|x|^{\alpha}\min\{d_n^{p-2}|\nabla w_n|^{p}, |c_n \nabla U_{\lambda_n}|^{p-2}|\nabla w_n|^{2}\} dx
    - o(d_n^2) \nonumber \\
    \geq & cd_n^{2} ,
    \end{align}
    \end{small}
    for some constant $c>0$, thus the conclusion (\ref{rtnmb2}) follows immediately.
    \end{proof}

\noindent{\bf Proof of Theorem \ref{thmprtp2}.} We argue by contradiction. In fact, if the theorem is false then there exists a sequence $\{u_n\}\subset \mathcal{D}^{1,p}_{\alpha,r}(\mathbb{R}^N)\backslash \mathcal{M}$ such that
    \begin{equation*}
    \frac{\|u_n\|^p-S_r\|u_n\|^p_*}{{\rm dist}(u_n,\mathcal{M})^2}\to 0,\quad \mbox{as}\quad n\to \infty.
    \end{equation*}
    By homogeneity, we can assume that $\|u_n\|=1$, and after selecting a subsequence we can assume that ${\rm dist}(u_n,\mathcal{M})\to \xi\in[0,1]$ since ${\rm dist}(u_n,\mathcal{M})=\inf_{c\in\mathbb{R}, \lambda>0}\|u_n-cU_{\lambda}\|\leq \|u_n\|$. If $\xi=0$, then we have a contradiction by Lemma \ref{lemma:rtnm2b2}. The other possibility only is that $\xi>0$, and we also deduce a contradiction by taking the same steps as in the proof of Theorem \ref{thmprtp}. Now, the proof of Theorem \ref{thmprtp2} is complete.
    \qed

\noindent{\bfseries Acknowledgements}

The research has been supported by National Natural Science Foundation of China (No. 11971392).

    \end{document}